\documentclass[12pt,leqno]{article}

\usepackage{mathrsfs}

\usepackage{amssymb, amsthm, amsfonts, amsxtra, amsmath}
\usepackage{latexsym}
\usepackage{mathabx}
\usepackage{stmaryrd} 
\usepackage{bbm}

\usepackage[all]{xy}
\usepackage{graphics}
\usepackage{latexsym}
\usepackage{makeidx}
\usepackage{rotating}

\theoremstyle{plain}

\newtheorem*{Theo}{Theorem}
\newtheorem{Theorem}{Theorem}[section]
\newtheorem{Lem}[Theorem]{Lemma}
\newtheorem{Prop}[Theorem]{Proposition}
\newtheorem{Cor}[Theorem]{Corollary}
\newtheorem{Def-Prop}[Theorem]{Definition-Proposition}

\theoremstyle{definition}
\newtheorem{Def}[Theorem]{Definition}
\newtheorem{Not}[Theorem]{Notation}
\newtheorem*{Defi}{Definition}
\newtheorem{Exa}[Theorem]{Example}
\newtheorem{Rem}[Theorem]{Remark}
\newtheorem*{Rema}{Remarks}

\date{}

\DeclareMathOperator{\Mor}{Mor}
\DeclareMathOperator{\End}{End}

\DeclareMathOperator{\sgn}{sgn}
\DeclareMathOperator{\Rep}{Rep}

\DeclareMathOperator{\Tr}{Tr}

\newcommand{\Circt}{\mathop{\ooalign{$\ovoid$\cr\hidewidth\raise-.05ex\hbox{$\scriptstyle\mathsf T\mkern3.5mu$}\cr}}} 
\newcommand{\smCirct}{\mathop{\ooalign{$\scriptstyle\ovoid$\cr\hidewidth\raise-.05ex\hbox{$\scriptscriptstyle\mathsf T\mkern2.8mu$}\cr}}}  

\DeclareMathOperator{\aA}{\mathscr{A}}
\DeclareMathOperator{\cC}{\mathrm{C}}
\newcommand{\SU}{\mathit{SU}} 
\newcommand{\SL}{\mathit{SL}}
\newcommand{\G}{\mathbb{G}}
\newcommand{\X}{\mathbb{X}}
\newcommand{\Y}{\mathbb{Y}}
\newcommand{\C}{\mathbb{C}}
\newcommand{\Z}{\mathbb{Z}}
\newcommand{\Zz}{Z}
\newcommand{\N}{\mathbb{N}}
\newcommand{\R}{\mathbb{R}}
\newcommand{\cat}[1]{\mathcal{#1}}

\newcommand{\Co}[1]{C(#1)}
\newcommand{\CoL}[1]{C(#1)} 
\mathchardef\mhyph="2D

\newcommand{\bimod}[2]{#1\mhyph\mathrm{mod}\mhyph#2}
\newcommand{\PW}{\mathrm{P}}

\newcommand{\Rmat}{\mathcal{R}}
\newcommand{\Jop}{\mathcal{J}}
\newcommand{\triv}{\mathrm{triv}}
\newcommand{\id}{\mathrm{id}} 
\newcommand{\Hsp}{\mathscr{H}}
\newcommand{\Ef}{\cat{E}_\textrm{f}}
\newcommand{\qn}[1]{\lbrack #1 \rbrack_q}
\newcommand{\aqn}[1]{\llbracket #1 \rrbracket_q}
\newcommand{\atn}[1]{\llbracket #1 \rrbracket_{t}}
\newcommand{\we}{W}
\newcommand{\twe}{T}

\newcommand{\absv}[1]{\left | #1 \right |}
\newcommand{\norm}[1]{\left \| #1 \right \|}
\newcommand{\half}{\frac{1}{2}}

\usepackage{tensor}

\newcommand{\ld}[2]{\tensor[_#1]#2{}}

\usepackage{parskip} 
\makeatletter 
\def\thm@space@setup{%
  \thm@preskip=\parskip \thm@postskip=0pt
}
\makeatother

\numberwithin{equation}{section} 

\begin{document}

\hyphenation{Wo-ro-no-wicz}

\title{Tannaka--Kre\u{\i}n duality for compact quantum homogeneous spaces. II. Classification of quantum homogeneous spaces for quantum $\SU(2)$}

\author{Kenny De Commer and Makoto Yamashita}
\maketitle

\begin{abstract}
\noindent We apply the Tannaka--Kre\u{\i}n duality theory for quantum homogeneous spaces, developed in the first part of this series of papers, to the case of the quantum $\SU(2)$ groups. We obtain a classification of their quantum homogeneous spaces in terms of weighted oriented graphs. The equivariant maps between these quantum homogeneous spaces can be characterized by certain quadratic equations associated with the braiding on the representations of $\SU_q(2)$. We show that, for $|q|$ close to 1, all quantum homogeneous spaces are realized by coideals.
\end{abstract}


AMS 2010 \emph{Mathematics subject classification}: 17B37; 20G42; 46L08


\pagestyle{myheadings}
\markright{De Commer and Yamashita, Quantum homogeneous spaces for quantum SU(2)}
\section*{Introduction}

This is a continuation of our previous paper on the Tannaka--Kre\u{\i}n duality for quantum homogeneous spaces.  In this paper, we apply our general machinery to the case of the `quantum $\SU(2)$ group' $\SU_q(2)$, where $0<|q|\leq 1$~\cite{Wor3,Wor4}.

The study of $\SU_q(2)$ and its actions on noncommutative spaces has special significance in the study of compact quantum groups. The quantum $\SU(2)$ group was and remains the prime example of the matrix quantum group theory initiated by S.L.~Woronowicz. Because of its close connection to the Drinfeld--Jimbo quantized universal enveloping algebra $\mathcal{U}_q(\mathfrak{sl}_2(\mathbb{C}))$, it gives rise to an interesting deformation of the finite-dimensional representation theory of $\mathfrak{sl}_2(\mathbb{C})$.

Our goal here is to study operator algebras which allow $\SU_q(2)$, for some $q$, as an ergodic symmetry group. We shall refer to such operator algebras as \emph{quantum homogeneous spaces} for $\SU_q(2)$. For $q=1$, A.~Wassermann's classification~\cite{Was1} of the ergodic actions of the classical Lie group $\SU(2)$ implied that it has only `classical' quantum homogeneous spaces, in the sense that any such action can be obtained by induction from a (projective) representation of a closed subgroup of $\SU(2)$. In particular, there are only countably many isomorphism classes of such actions.

The situation for $\SU_q(2)$ with $|q|<1$ is drastically different. First, P.~Podle\'s showed that there is a continuum of non-isomorphic `quantum $2$-spheres' over $\SU_q(2)$ which do not necessarily correspond to quantum subgroups.  R.~Tomatsu~\cite{Tom1} extended this by classifying all quantum homogeneous spaces over $\SU_q(2)$ which are of coideal type.  In another direction, J.~Bichon, A.~de~Rijdt and S.~Vaes~\cite{BDV1} discovered that, for $|q|$ small enough, a family of non-coideal type quantum homogeneous spaces exists, admitting large multiplicities for spectral subspaces. This diversity makes it important to understand the general structure of a quantum homogeneous space for $\SU_q(2)$, and this will be the main focus point of this paper.

In \cite{DCY2}, it was shown that the quantum homogeneous spaces over a compact quantum group $\G$ can be classified by the indecomposable semi-simple module $\cC^*$-categories over the tensor $\cC^*$-category of finite-dimensional unitary representations of $\G$. In a purely algebraic setting, \emph{finite} indecomposable semi-simple module categories for $\Rep(\SL_q(2))$ were classified in \cite{Eti1}. However, in the case of $\SU_q(2)$, the extra positivity conditions which are present mean that we have to further refine the classification of purely algebraic module categories over $\Rep(\SU_q(2))$. On the other hand, we can treat the case of module categories having infinitely many irreducible objects (which was recently treated also in the purely algebraic case in \cite{Gel1}). Secondly, because of the above positivity, we can label the parameter space of our classification in a more convenient way since all matrices concerned can be diagonalized. We first introduce some terminology.

\begin{Defi}
An \emph{oriented graph} $\Gamma$ consists of a countable set $\Gamma^{(0)}$, the set of \emph{vertices}, and a countable set $\Gamma^{(1)}$, the set of \emph{(oriented)} edges, together with two maps $\Gamma^{(1)} \overset{s}{\underset{t}{\rightrightarrows}} \Gamma^{(0)}$ called the \emph{\underline{s}ource} and \emph{\underline{t}arget} map. In particular, we allow an infinite number of vertices and edges, and multiple edges as well as loops at vertices.  A \emph{cost} (or \emph{weight}) on an oriented graph is a positive real valued function $\we\colon \Gamma^{(1)}\rightarrow \mathbb{R}^+$ on the edge set.  When $v$ is a vertex, the \emph{source cost} $\we(v)\in \lbrack 0,+\infty\rbrack$ is the sum of the costs of all the edges leaving from $v$. We call an oriented graph \emph{symmetric} if it can be equipped with an involution $e\rightarrow\bar{e}$ on the edge set which interchanges source and target vertex of each edge. We call a cost on a symmetric graph \emph{balanced} if one can choose an involution satisfying $\we(e)\we(\bar{e})=1$.
\end{Defi}

\begin{Defi}
Let $\Gamma$ be an oriented graph, and let $\twe$ be a nonzero real number.  A \emph{fair and balanced $\twe$-cost} on $\Gamma$ is a balanced cost on $\Gamma$ such that the source cost at any vertex is equal to $\absv{\twe}$, and with an even number of loops at each vertex for the case $\twe > 0$.  A graph with a fair and balanced $\twe$-cost is also called a \emph{fair and balanced $\twe$-graph}.
\end{Defi}

\begin{Rema}\begin{enumerate}\item  We stress that in the above definitions, the involution is \emph{not} part of the data set. More precisely, an isomorphism between two fair and balanced $\twe$-graphs $(\Gamma_i,\we_i)$ is simply a couple of bijective maps $\phi^{(k)}\colon\Gamma_1^{(k)}\rightarrow \Gamma_2^{(k)}$ which intertwine the source, target and weight maps.
\item It is easily seen that the degree $\deg(\Gamma) = \sup_v \#\{e\mid s(e) = v\}$ of a fair and balanced $\twe$-graph $\Gamma$ is bounded. More precisely, $\deg(\Gamma)\leq \twe^2$.
\item Another way to look at a fair and balanced cost is by considering the associated `random walk', which assigns to each edge the probability
\[
P(e) =  \frac{\we(e)}{\sum_{f,s(e)=s(f)} \we(f)}= \frac{\we(e)}{|\twe|}.
\]
Then we see that, having fixed $\twe$, the random walk $P$ completely determines $\we$, and the only condition imposed on $P$ is that $P(e)P(\bar{e}) = \twe^{-2}$ (which one can could call \emph{$\twe^{-2}$-reciprocality}). However, we chose to work with the cost $\we$ since the condition $\we(e) \we(\bar{e}) = 1$ better reflects the crucial r\^ole of the inversion, and since the probabilistic nature of the formulation is not fundamentally used yet in the present work.
\end{enumerate}
\end{Rema}

\begin{Theo}[Theorem~\ref{TheoF3}]
For $0<|q|\leq 1$, the quantum homogeneous spaces over $\SU_q(2)$ are classified, up to equivariant Morita equivalence, by connected fair and balanced $q+q^{-1}$-graphs.
\end{Theo}

The connected fair and balanced $2$-graphs are easy to classify in a direct way. Hence we obtain in particular a more conceptual proof (Proposition~\ref{PropW2}) of the above mentioned result of Wassermann for $\SU(2)$.  In the quantized setting, when $|q|$ is close enough to $1$, our classification implies that any quantum homogeneous space is of coideal type (see Theorem~\ref{TheoBound}), naturally generalizing the classical case $q=1$. On the other hand, as $q$ gets smaller, there arises a profusion of ergodic actions, since any symmetric graph $\Gamma$ with bounded degree (and an even number of loops at each vertex) admits at least one fair and balanced $\twe$-cost for $\twe= \pm \|\Gamma\|$.

Here is a short summary of the contents of this paper. In the \emph{first section}, we recall the definition of the quantum $\SU(2)$ groups, and briefly discuss some of the main results of \cite{DCY2} for these particular quantum groups. In the \emph{second section}, we prove the theorem stated above. The main observation here is that the representation category of $\SU_q(2)$ is in essence the Temperley--Lieb category (\cite{Wor4}, \cite{Ban2}), whose universal property can be exploited to encode semi-simple $\cC^*$-module categories over $\Rep(\SU_q(2))$ in terms of the combinatorial structure of weighted graphs. In the \emph{third section}, we give some more information on connected fair and balanced $\twe$-graphs. In the \emph{fourth section}, we give a more concrete classification of $SU_q(2)$-homogeneous spaces in the region $|q| \in (1-\varepsilon,1\rbrack$, for some small $\varepsilon$. In the \emph{fifth section}, a connected fair and balanced $q+q^{-1}$-graph $(\Gamma,\we)$ is shown to give rise to a particular, non-ergodic action on a $^*$-algebra $\aA$, having the associated ergodic actions as its corners. The $^*$-algebra $\aA$ can be explicitly given in terms of generators and relations determined by the weighted graph $(\Gamma,w)$. In the \emph{sixth section}, we show that equivariant maps between the quantum homogeneous spaces over $\SU_q(2)$ can be determined by certain quadratic equations on isometries associated with the weighted graphs.  Finally, in the \emph{seventh section}, we use some of the above results to determine  structural properties of the C$^*$-algebra underlying an ergodic action of $\SU_q(2)$: we show that the C$^*$-algebra is of type I if and only if it is of coideal type, and that the $K$-groups can be computed from the associated graph, using the resolution of the Baum--Connes conjecture for $SU_q(2)$ obtained by Voigt.

\paragraph{Acknowledgments} We would again like to thank T. Banica, R. Conti, Y. Kawahigashi, R. Meyer, S. Neshveyev, V. Ostrik, C. Pinzari, N. Snyder, S. Vaes, C. Voigt, and particularly R. Tomatsu, for valuable discussions both on the content and on the presentation of this paper.   M.Y. is supported by the Danish National Research Foundation through the Centre for Symmetry and Deformation (DNRF92).

\section{Preliminaries}

Throughout the paper $q$ is a real number with $0<|q|\leq 1$. The $q$-integer $(q^{-n} - q^{n})/(q^{-1} - q)$ is denoted by $\qn{n}$. We shall frequently employ its absolute value $\absv{\qn{n}}$, which we will denote by $\aqn{n}$.

The cyclic group of order $n$ is denoted by $\Zz_n$.

We will freely use notation and terminology as introduced in~\cite{DCY2}. Although we shall briefly recall these notations when necessary, the reader is nevertheless strongly encouraged to consult \cite{DCY2} beforehand, particularly its Section~2 
and the Appendix.

\begin{Def}\label{DefSUq2}
The $^*$-algebra $\PW(\SU_{q}(2))$ of \emph{regular functions on the compact quantum group $\SU_{q}(2)$} is the universal unital $^*$-algebra over $\C$ with generators $u_{ij}$, where $i,j\in \{1,2\}$, subject to the relations
\begin{equation}\label{EqSUq2DefGenRel1}
\begin{pmatrix}
u_{1 1}^* & u_{1 2}^*\\
u_{2 1}^* & u_{2 2}^*
\end{pmatrix}
=
\begin{pmatrix}
u_{2 2} & -q u_{2 1}\\
-q^{-1} u_{1 2} & u_{1 1}
\end{pmatrix}
\end{equation}
and
\begin{equation}\label{EqSUq2DefGenRel2}
\begin{pmatrix}
u_{1 1}^* & u_{2 1}^*\\
u_{1 2}^* & u_{2 2}^*
\end{pmatrix}
\begin{pmatrix}
u_{1 1} & u_{1 2}\\
u_{2 1} & u_{2 2}
\end{pmatrix}
=
\begin{pmatrix}
u_{1 1} & u_{1 2}\\
u_{2 1} & u_{2 2}
\end{pmatrix}
\begin{pmatrix}
u_{1 1}^* & u_{2 1}^*\\
u_{1 2}^* & u_{2 2}^*
\end{pmatrix}
=
\begin{pmatrix}
1 & 0\\
0 & 1
\end{pmatrix}.
\end{equation}
It is a Hopf $^*$-algebra whose coproduct is defined by
\[
\Delta(u_{i j}) = u_{i 1} \otimes u_{1 j} + u_{i 2} \otimes u_{2 j}
\]
for any $i, j \in \{1, 2\}$.
\end{Def}

Consider the matrix  $u \in M_2(\PW(\SU_q(2)))$ with components $(u_{i j})_{i j}$, and the matrix $\bar{u}$ with components $(u_{i j}^*)_{i j}$.  If we put
\begin{equation}\label{EqFmat}
F = \begin{pmatrix}
0 & |q|^{\half} \\
-\sgn(q) |q|^{-\half} & 0
\end{pmatrix},
\end{equation}
the defining relations~\eqref{EqSUq2DefGenRel1} and~\eqref{EqSUq2DefGenRel2} simply say that $u$ is unitary and $\bar{u} = F^{-1} u F$.

\begin{Rem}
At the classical limit $q=1$, we obtain the Hopf $^*$-algebra of the matrix coefficients of finite-dimensional unitary representations of $\SU(2)$.  On the other extreme, when $q = -1$, the quantum group $\SU_{-1}(2)$ can be interpreted as the free orthogonal quantum group $O_2^+$~(\cite{Ban3}).
\end{Rem}

Let $\frac{1}{2}\N$ be the set $\{0, \half, 1, \frac{3}{2}, \ldots \}$ of nonnegative half integers.  The highest weight theory gives a labeling of isomorphism classes of irreducible objects in $\Rep(\SU_q(2))$ by $\frac{1}{2}\N$.  That is, for each $n \in \half \N$, we have an irreducible representation $u_n$, unique up to isomorphism, having the classical dimension $2 n + 1$, and the quantum dimension $\aqn{2n+1}$.  The tensor product of two such representations decomposes the same way as in the classical case:
\[
u_m \Circt u_n = u_{\absv{m - n}} \oplus u_{\absv{m - n} + 1} \oplus \cdots \oplus u_{m + n}.
\]

The first nontrivial irreducible representation $u = u_{1/2}$ is of special importance.  It can be realized on a $2$-dimensional Hilbert space $\Hsp_{1/2}$ with an orthonormal basis $(e_1, e_2)$, endowed with the comodule structure $\delta(e_i) = e_1 \otimes u_{1 i} + e_2 \otimes u_{2 i}$.  Thus, the generators $(u_{i j})_{i j}$ of $\PW(\SU_q(2))$ are precisely the matrix coefficients for this choice of basis.

The matrix $F$ is related to the duality for $u$.  Namely, the linear map
\begin{equation}\label{EqRmatForHalfSpin}
R_u \colon \C \rightarrow \Hsp_{1/2} \otimes \Hsp_{1/2}, \quad \lambda \mapsto \lambda \left ( - \sgn(q) \absv{q}^{-1/2} e_1\otimes e_2 + \absv{q}^{1/2}e_2\otimes e_1 \right ).
\end{equation}
is an intertwiner from $u_o$ to $u \Circt u$, and satisfies the equation
\begin{align}\label{EqQFundEqForHalfSpin}
(R_u^*\otimes 1)(1\otimes R_u) &= -\sgn(q) \id_{u},&
R_u^* R_u &= \aqn{2} \id_o.
\end{align}
In particular, it follows that $u$ itself is a dual object for $u$, and $(R_u, -\sgn(q) R_u)$ gives a pair of duality morphisms. We then have that $Fe_i = (e_i^*\otimes \id)R_u$ for $i\in \{1,2\}$.

\begin{Def} Let $\cat{C}$ be a strict tensor $\cC^*$-category.  We call a couple $(x,\Rmat)$ a \emph{$q$-fundamental solution} in $\cat{C}$ if $x$ is an object in $\cat{C}$, and $\Rmat\in \Mor(\mathbbm{1}_{\cat{C}},x\otimes x)$ satisfies the condition~\eqref{EqQFundEqForHalfSpin} for $(x, \Rmat)$ in place of $(u, R_u)$.

We call two $q$-fundamental solutions $(x,\Rmat)$ and $(y,\mathcal{S})$ \emph{equivalent} if there exists a unitary morphism $U$ from $x$ to $y$ such that $\mathcal{S} = (U\otimes U)\Rmat$.
\end{Def}

By the above definition, if $(x, \Rmat)$ is a $q$-fundamental solution, one has $\bar{x} = x$ by the duality morphism pair $(\Rmat, -\sgn(q) \Rmat)$.

The following theorem is then well-known in one or another form. It states that the representation category of $\SU_q(2)$, which can also be realized as the Karoubi envelope of the Temperley--Lieb category, is the universal tensor $\cC^*$-category generated by a $q$-fundamental solution.

\begin{Theorem}[Chapter XII of \cite{Tur1}, Remark 2.2.4 in \cite{Eti1}, Lemma~6.1 of~\cite{Yam1}, Theorem 6.2 of~\cite{Pin2}, Sections 6--8 of~\cite{Pin3}] \label{ThmGen} Let $\cat{C}$ be a strict tensor $\cC^*$-category, and let $(x,\Rmat)$ be a $q$-fundamental solution in $\cat{C}$. Then there exists a unique strict tensor $\cC^*$-functor \[\mathcal{F}\colon \Rep(\SU_q(2))\rightarrow \cat{C}\] such that $\mathcal{F} u = x$ and such that $\Rmat$ equals $\Rmat_{\mathcal{F}} = \mathcal{F}(R_u)$.

Moreover, two strict tensor $\cC^*$-functors $\mathcal{F}$ and $\mathcal{G}$ from $\Rep(\SU_q(2))$ into $\cat{C}$ are tensor equivalent if and only if the $q$-fundamental solutions $(\mathcal{F} u ,\Rmat_{\mathcal{F}})$ and $(\mathcal{G} u,\Rmat_{\mathcal{G}})$ are equivalent.
\end{Theorem}

\begin{proof} The first part of the theorem as stated is found most explicitly, in the $^*$-setting, in \cite{Pin3}. However, one can easily modify the version in \cite{Tur1} to accommodate for the $^*$-structure.  The second part is not found explicitly in these references, but can be deduced easily from their techniques: given an equivalence $U$ between $(\mathcal{F}u,\Rmat_{\mathcal{F}})$ and $(\mathcal{G}u,\mathcal{R}_{\mathcal{G}})$, we get a map \[\tau_{u^{\otimes n}}\colon \mathcal{F}(u^{\otimes n}) = \mathcal{F}(u)^{\otimes n} \overset{U^{\otimes n}}{\rightarrow} \mathcal{G}(u)^{\otimes n}= \mathcal{G}(u^{\otimes n}).\] Using that morphisms between $u^{\otimes n}$ and $u^{\otimes m}$ can be expressed as algebraic combinations of $R_u$ and $R_u^*$, it follows that the above map is natural. It can then be extended to the desired tensor equivalence.
\end{proof}

\begin{Rem} The strictness assumption is not essential, since, in this particular case, any strong tensor $\cC^*$-functor from $\Rep(\SU_q(2))$ to $\cat{C}$ is equivalent with a strict one.
\end{Rem}

Let $J$ be an index set. We recall that $\cat{E}^J_f$ is the rigid tensor $\cC^*$-category of $J\times J$-graded Hilbert spaces $\mathscr{H} = \oplus_{v,w\in J} \mathscr{H}_{vw}$ such that $\sup_v \sum_w (\dim(\Hsp_{vw})+\dim(\Hsp_{wv}))<\infty$. Applying Theorem~\ref{ThmGen} to this category (which for all intents and purposes may be assumed strict), we obtain the following proposition.

\begin{Prop}\label{PropEq}
A strict tensor $\cC^*$-functor $\mathcal{F}\colon \Rep(\SU_q(2)) \rightarrow \Ef^{J}$ is completely determined by the values of $\mathcal{F}u$ and $\Rmat_{\mathcal{F}} = \mathcal{F}(R_u)$.  These data can be any pair $(\Hsp, \Rmat)$ such that $\Hsp$ is an object in $\Ef^{J}$ and $\Rmat = \oplus_v \sum_w \Rmat_{vw}$ is a family of maps
\[
\Rmat_{v w}\colon \mathbb{C}\rightarrow \Hsp_{v w}\otimes \Hsp_{w v}
\]
satisfying
\begin{enumerate}
\item
$(\Rmat_{v w}^*\otimes \id) (\id \otimes \Rmat_{w v}) = -\sgn(q) \id$ on $\Hsp_{v w}$, and
\item
$\sum_w \Rmat_{v w}^*\Rmat_{v w} = \aqn{2}$, for all $v$.
\end{enumerate}

Two $q$-fundamental solutions $(\Hsp,\Rmat)$ and $(\mathscr{G},\mathcal{S})$ respectively in $\Ef^J$ and $\Ef^{J'}$ give isomorphic $\Rep(\SU_q(2))$-module categories if and only if there exists a bijection $\phi\colon J'\rightarrow J$ and unitaries $U_{v w}\colon \mathscr{G}_{\phi(v)\phi(w)}\rightarrow \Hsp_{v w}$ such that $\Rmat_{v w} = (U_{v w}\otimes U_{w v})\mathcal{S}_{\phi(v)\phi(w)}$.
\end{Prop}

\begin{Rem}  The module category associated with such an $\mathcal{F}$ is connected if and only if the grading on $\Hsp$ can not be decomposed into two separate blocks.
\end{Rem}

The above conditions on the $\Rmat_{v w}$ imply that $\Hsp_{v w}$ and $\Hsp_{w v}$ have the same dimension for any pair $(v, w)$.  We can also represent these $\Rmat_{v w}$ as anti-linear maps $\Jop_{v w}\colon \Hsp_{v w}\rightarrow \Hsp_{w v}$ defined by
\begin{equation}\label{EqRandJ}
\Jop_{v w}\xi  =(\xi^*\otimes \id)(\Rmat_{v w}(1)),
\end{equation}
where $\xi^*(\eta) = \langle \xi,\eta\rangle$.

\begin{Prop}[Cf. \cite{DCY2}, Lemma A.3.2] \label{PropJCondi}
The operators $(\Jop_{v w})_{v, w}$ satisfy $\Jop_{w v}\Jop_{v w} = -\sgn(q) \id$ and, for any fixed $v$, $\sum_{w} \Tr(\Jop_{v w}^* \Jop_{v w}) = \aqn{2}$.
\end{Prop}

\begin{proof}
Let us first check $\Jop_{w v}\Jop_{v w} = -\sgn(q) \id$.  When $\xi\in \Hsp_{v w}$, unwinding the definition, we have
\[
\Jop_{w v}\Jop_{v w} \xi = \Jop_{w v}((\xi^*\otimes \id)\Rmat_{v w}(1)) = (\Rmat_{v w}^*\otimes \id)(\xi\otimes \Rmat_{w v}(1)).
\]
Using the first equality in~\eqref{EqQFundEqForHalfSpin}, the right hand side is equal to $ -\sgn(q)\xi$.

We next verify the condition on the traces.  Choose orthonormal bases $(\xi_i^{v w})_i$ of the Hilbert spaces $\Hsp_{v w}$. Then
\begin{align*}
\sum_w \Rmat_{v w}^*\Rmat_{v w} &= \sum_{w}\sum_{i} ((\xi_i^{v w *}\otimes \id)\Rmat_{v w})^* ((\xi^{v w *}_i \otimes \id)\Rmat_{v w}) \\
&= \sum_w \sum_{i} \langle \Jop_{v w} \xi_i^{v w},\Jop_{v w}\xi^{v w}_i\rangle \\
&= \sum_w \Tr(\Jop_{v w}^*\Jop_{v w}).
\end{align*}
By the second relation in Proposition \ref{PropEq}, we have $\sum_{w} \Tr(\Jop_{v w}^* \Jop_{v w}) = \aqn{2}$.
\end{proof}

Conversely, if a collection of anti-linear operators $\Jop_{v w}\colon \Hsp_{v w}\rightarrow \Hsp_{w v}$ satisfies the conditions of the proposition above, the family $(\Rmat_{v w})_{v w}$ defined by~\eqref{EqRandJ} gives a $q$-fundamental solution.  Two such collections $(\Hsp_{v w},\Jop_{v w})$ and $(\mathscr{G}_{v w},\mathcal{I}_{v w})$ are then equivalent if and only if there exists a bijection $\phi\colon J'\rightarrow J$ and unitaries $U_{v w}\colon \mathscr{G}_{\phi(v)\phi(w)}\rightarrow \Hsp_{v w}$ such that $\Jop_{v w} = U_{v w}\mathcal{I}_{\phi(v),\phi(w)}U_{v w}^*$.

\section{Classification of $\SU_q(2)$-homogeneous spaces by graphs}

We use the results from the previous section to classify the quantum homogeneous spaces of $SU_q(2)$ in terms of weighted graphs.

\begin{Not}
Let $J$ be a countable set, and $q\in \mathbb{R}$ with $0<|q|\leq 1$.  We let $\mathcal{T}^J_q$ denote the collection of $q$-fundamental solutions $(\Hsp,\Rmat)$ in $\Ef^J$.
\end{Not}

\begin{Not}
Let $(\Hsp, \Rmat) \in \mathcal{T}^J_q$, and let  $\Jop_{v w}$ be the associated anti-linear operators as in \eqref{EqRandJ}.  We let $(\lambda^{(vw)}_k)_k$ denote the eigenvalues of $\Jop_{v w}^* \Jop_{v w}$, counted with multiplicity.

We define $W(\Hsp,\Rmat)$ as the oriented graph with cost which has the vertex set $J$, and $\dim(\Hsp_{v w})$ arrows from $v$ to $w$ with costs $(\lambda^{(v w)}_k)_k$.
\end{Not}

In the following proposition, we provide a fundamental domain for the equivalence classes of $q$-fundamental solutions in $\Ef^J$ (cf. the discussion after Theorem 5.5 of~\cite{BDV1}).

\begin{Prop}\label{LemGrphFromQFund}
Let $J$ and $J'$ be countable sets.  When $(\Hsp, \Rmat) \in \mathcal{T}^J_q$ and $(\mathscr{G}, \mathcal{S}) \in \mathcal{T}^{J'}_q$, the associated weighted graphs $W(\Hsp, \Rmat)$ and $W(\mathscr{G}, \mathcal{S})$ are isomorphic if and only if the induced tensor functors from $\Rep(\SU_q(2))$ into $\Ef^J$ and $\Ef^{J'}$ are tensor equivalent.

Moreover, any $W(\Hsp, \Rmat)$ is a fair and balanced $\qn{2}$-graph, and all fair and balanced $\qn{2}$-graphs arise in this way.
\end{Prop}

\begin{proof} First we remark that, by Proposition \ref{PropEq}, equivalence of tensor functors can be replaced by equivalence of $q$-fundamental solutions.

Let now $(\Hsp,\Rmat)$ be a $q$-fundamental solution with associated anti-linear maps $\Jop_{v w}$.  First, the sum of weights on edges starting from a vertex $v$ is given by $\sum_{w} \Tr(\Jop_{v w}^* \Jop_{v w})$.  Thus, Proposition~\ref{PropJCondi} implies that $W(\Hsp, \Rmat)$ has the constant source weight $\aqn{2}$.

Let us show that the cost is balanced. If there are no edges from $v$ to $w$, this means $\Jop_{v w}=0$. Since $\Jop_{v w}\Jop_{w v} = -\sgn(q)\id$, it follows that also $\Jop_{w v}=0$, and so there are no edges from $w$ to $v$. Assume now that $\Jop_{v w}\neq0$. We consider the left polar decomposition $\Jop_{v w} = J_{v w} P_{v w}$, so that $J_{v w}$ is an isometric anti-linear map and $P_{v w}$ is a positive linear map. From the condition $\Jop_{v w}\Jop_{w v} = -\sgn(q)$, we know that $J_{v w}$ is an anti-unitary, that $P_{v w}$ is invertible and that $P_{v w}^{-1} (-\sgn(q) J_{v w}^*)$ is the right polar decomposition of $\Jop_{w v}$.  By the uniqueness of the (left) polar decomposition, we obtain $J_{w v} = -\sgn(q)J_{v w}^*$ and $P_{w v} = J_{v w} P_{v w}^{-1} J_{v w}^*$. It follows that, with preservation of multiplicities, we have
\begin{equation*}
\mathrm{Spec}(\Jop_{v w}^*\Jop_{v w}) = (\mathrm{Spec}(\Jop_{w v}^*\Jop_{w v}))^{-1}.
\end{equation*}
This shows that $W(\Hsp,\Rmat)$ has an involution such that the cost becomes balanced.  Thus, $W(\Hsp, \Rmat)$ is a fair and balanced $\aqn{2}$-graph.  Moreover, because $\Jop_{v v}^2= -\sgn(q)$, the multiplicity of $1$ in $\Jop_{v v}^*\Jop_{v v}$ is even in case $q>0$ and the involution can be made fixed point free. We conclude that $W(\Hsp, \Rmat)$ is a fair and balanced $\qn{2}$-graph.

Now, two $q$-fundamental solutions $(\Hsp,\Rmat)$ and $(\mathscr{G},\mathcal{S})$ with associated anti-linear maps $\Jop_{v w}$ and $\mathcal{I}_{v w}$ are equivalent if and only if there exists a bijection $\phi\colon J'\rightarrow J$ and unitaries $U_{v w}\colon \mathscr{G}_{\phi(v)\phi(w)}\rightarrow \Hsp_{v w}$ such that $\Jop_{v w} = U_{v w}\mathcal{I}_{\phi(v)\phi(w)}U_{v w}^*$. Hence we see that $W(\Hsp,\Rmat)$ does not depend on the equivalence class of $(\Hsp,\Rmat)$, since this simply corresponds to relabeling of the vertices.

Conversely, let $(\Gamma,\we)$ be a fair and balanced $\qn{2}$-graph, with a labeling of its vertices by $J$. Choose an involution on the edge set as in the balancedness condition, fixed point free in case $q>0$. Choose an arbitrary function $\rho\colon \Gamma^{(1)}\rightarrow \{-1,1\}$ such that $\rho(e)\rho(\bar{e}) = -\sgn(q)$ for all edges $e$. Such a function exists by the stated assumption in the $q>0$ case. Let $\Hsp_{v w}$ be the vector space spanned by the edges in $\Gamma^{(1)}$ having source $v$ and range $w$, and make it into a Hilbert space by making these edges an orthonormal basis. Finally, let $\mathscr{H}$ be the Hilbert space direct sum of the $\Hsp_{v w}$ with the obvious $J\times J$-grading. Note that by balancedness and fairness, the number of edges coming out of or going into any given vertex is uniformly bounded, so that $\bigoplus_{v,w}\Hsp_{v w}$ is an element of $\cat{E}_f^J$.

Define then $\Jop_{v w}\colon\Hsp_{v w}\rightarrow \Hsp_{w v}$ as the unique anti-linear operator taking the standard basis vector $e$ to $\rho(e) \we(e)^{1/2} \bar{e}$. It is clear, by construction, that these operators satisfy the equations in Proposition \ref{PropJCondi}. Hence they give a $q$-fundamental solution in $\cat{E}_f^J$.

It is left to show that these two maps are inverses of each other. The only difficulty may consist in showing that an arbitrary solution $(\Hsp,\Rmat)$ is isomorphic to the solution constructed from $W(\Hsp,\Rmat)$ as in the previous paragraph. However, choosing bases in the $\Hsp_{v w}$ which diagonalise the $\Jop_{v w}^*\Jop_{v w}$, we get immediately a choice for the function $\rho$. It is then easy to construct the desired isomorphism.
\end{proof}

\begin{Theorem}\label{TheoF3} For $0< \absv{q} \leq 1$, the quantum homogeneous spaces over $\SU_q(2)$ are classified, up to equivariant Morita equivalence, by connected fair and balanced $q+q^{-1}$-graphs.
\end{Theorem}

\begin{proof}
The theorem follows by combining Proposition~\ref{LemGrphFromQFund} with~\cite[Theorem~6.4]
{DCY2}. Since the powers of $u$ generate $\Rep(\SU_q(2))$, the graph $W(\Hsp, \Rmat)$ is connected if and only if the associated module category is connected, in which case the index set must necessarily be countable.
\end{proof}

\section{Connected fair and balanced graphs}

In the following, each a priori unoriented graph will be interpreted as a symmetric graph in the obvious way.


\subsection{Examples of fair and balanced graphs from Frobenius--Perron theory}

\begin{Prop}[c.f. Theorem~3.5 of~\cite{Eti1}]\label{PropExCost}
Let $\Gamma$ be a connected symmetric graph. Then $\Gamma$ admits a fair and balanced $\twe$-cost for some $\twe<0$ if and only if it has finite degree, i.e.~the number of edges emanating from a vertex is uniformly bounded. It admits a fair and balanced $\twe$-cost for some $\twe>0$ if in addition the number of loops at each vertex is even.
\end{Prop}

\begin{proof} As already remarked in the Introduction and the proof of Proposition \ref{LemGrphFromQFund}, the existence of a fair and balanced $\twe$-cost necessarily implies that the degree of the graph is finite.

Conversely, if $\Gamma$ has finite degree, the norm $\|\Gamma\|$ of the adjacency matrix $A(\Gamma)$ of $\Gamma$ is finite. From classical Frobenius--Perron theory (for finite graphs) and \cite{Pru1} (for infinite graphs), it follows that we can find a formal eigenvector $c$ of $A(\Gamma)$ at eigenvalue $\|\Gamma\|$, such that $c_v>0$ for all vertices $v$. If then $e$ is any edge from $v$ to $w$, associate with it the cost $c_w/c_v>0$. It is clear that this cost is balanced. Moreover, the eigenvector property implies that it gives a fair and balanced $\twe$-cost for $\twe=-\|\Gamma\|$. Of course, this is a fair and balanced $\|\Gamma\|$-graph if the loops at vertices are even in number.
\end{proof}

\begin{Rem}\label{RemPFArg}
\begin{enumerate}
\item When $\Gamma$ is infinite, the same result of~\cite{Pru1} gives that we have fair and balanced $\twe$-costs for any $\twe$ with $|\twe|\geq \|\Gamma\|$.
\item When $\Gamma$ is a tree, any fair and balanced cost on $\Gamma$ must arise from a Frobenius--Perron eigenvector. Indeed, choosing a root for the tree, one can assign to any vertex the product of the costs of the edges in the unique minimal path from the root to that vertex, which is easily seen to give an eigenvector for the adjacency matrix with strictly positive entries. In general, as in the Examples~\ref{ExaCyclicSubgrp} below, other finite graphs might admit fair and balanced costs which are not induced by Frobenius--Perron vectors. See also the discussion in Theorem~3.5 of~\cite{Eti1}, whose techniques are also applicable in the positive setting.
\end{enumerate}
\end{Rem}

\begin{Cor}
Any connected symmetric graph $\Gamma$ of bounded degree and norm $\|\Gamma\|\geq 2$ arises as the graph of an ergodic action of $\SU_q(2)$, for at least one value of $q$.
\end{Cor}

The countable set of graphs with norm $< 2$ has to be excluded of course, since it correspond to the root of unity case $\twe= q+q^{-1}$ with $q\in \{e^{\frac{i\pi}{n}}\mid n\geq 3\}$, for which the operator algebraic $SU_q(2)$ is not defined. Cf. \cite{Ocn1}.

\begin{Rem}
The case of finite trees shows that there exist `isolated' examples of quantum homogeneous spaces for $\SU_q(2)$ which only appear at one particular value of $|q|$, c.f.~ the \emph{superrigid} graphs in~\cite{Eti1}.  The matrices associated with these graphs tend to have non-integer norms, and hence the corresponding quantum homogeneous space algebras cannot be embedded equivariantly into full quantum multiplicity ones by~\cite[Proposition~7.5]
{DCY2}. Cf.~Corollary 4.2 of \cite{Pin4} for a related result.
\end{Rem}

\begin{Def}
Let $\Gamma$ be an oriented graph, and let $n$ be a positive integer.  We let $\Gamma^{(n)}$ denote the graph which has the same vertices as $\Gamma$, and the paths of length $n$ in $\Gamma$ as its edges.
\end{Def}

\begin{Prop}\label{PropNStepGraph}
Let $\twe$ be a nonzero real number, and let $\Gamma$ be a connected symmetric graph endowed with a fair and balanced $\twe$-cost $\we$.  Then, for any positive integer $n$, the graph $\Gamma^{(n)}$ admits a fair and balanced $((-1)^{n+1}\twe^n)$-cost.
\end{Prop}

\begin{proof}
When $(e_1, e_2, \ldots, e_n)$ is an $n$-tuple of composable edges in $\Gamma$, we define the weight of the corresponding edge in $\Gamma^{(n)}$ to be $\prod_{j=1}^n \we(e_j)$.  This way, $\Gamma^{(n)}$ admits the constant source weight $\absv{\twe}^n$.  We can also define an involution on $\Gamma^{(n)}$ by sending $(e_1, \ldots, e_n)$ to $(\bar{e}_n, \ldots, \bar{e}_1)$, and make it a fair and balanced $(-\absv{\twe}^n)$-cost.  If the originally chosen involution $e\rightarrow \bar{e}$ was free and $n$ is odd, the involution $(\bar{e}_n, \ldots, \bar{e}_1)$ differs from $(e_1, \ldots, e_n)$ at least in the middle, so that we have in fact a $(-(-\twe)^n)$-cost.
\end{proof}

We note that, when $n$ is even, connectedness of $\Gamma$ does not imply connectedness of $\Gamma^{(n)}$.

\begin{Rem}
We give an interpretation of the above proposition in terms of the tensor $\cC^*$-functors between the categories $\Rep(\SU_q(2))$ for different values of $q$.  Namely, using \eqref{EqQFundEqForHalfSpin} successively, one sees that the morphism
\[
R_{u^{\smCirct n}} = (\id_{u^{\smCirct n-1}} \otimes R_u \otimes \id_{u^{\smCirct n-1}}) (\id_{u^{\smCirct n-2}} \otimes R_u \otimes \id_{u^{\smCirct n-2}}) \cdots R_u \in \Mor(u_o, u^{\smCirct n} \Circt u^{\smCirct n})
\]
satisfies
\begin{align*}
(R_{u^{\smCirct n}}^* \otimes \id_{u^{\smCirct n}}) (\id_{u^{\smCirct n}} \otimes R_{u^{\smCirct n}}) &= (-\sgn(q))^n \id_{u^{\smCirct n}},&
R_{u^{\smCirct n}}^* R_{u^{\smCirct n}} &= \aqn{2}^n.
\end{align*}
Thus, for $q'$ satisfying $\sgn{q'} = (-\sgn(q))^n$ and $\llbracket 2 \rrbracket_{q'} = \aqn{2}^n$, we get a $q'$-fundamental solution $(u^{u^{\smCirct n}}, R_{u^{\smCirct n}})$ in $\Rep(\SU_q(2))$. This defines a tensor $\cC^*$-functor $\Rep(\SU_{q'}(2)) \rightarrow \Rep(\SU_q(2))$, and any semi-simple module $\cC^*$-category over $\Rep(\SU_q(2))$ can be considered as one over $\Rep(\SU_{q'}(2))$.  The correspondence of fair and balanced graphs is as described in Proposition~\ref{PropNStepGraph}.
\end{Rem}

\begin{Prop}\label{PropCostVSGrphNorm}
Let $\Gamma = (V, E, s, t)$ be a connected graph.  If there is a fair and balanced $\twe$-cost $\we$ on $\Gamma$, then one has $\norm{\Gamma} \le \absv{\twe}$.  When the equality holds, the function $\we$ is constant on the set $E_{v,w}=\{ e \in E \mid s(e) = v, t(e) = w\}$ for any $(v, w) \in V \times V$.
\end{Prop}

\begin{proof}
Suppose that there is a fair and balanced $\twe$-cost on $\Gamma$.  Consider an operator $B$ from $\ell^2 V$ to $\ell^2 E$ defined by
\[
B(v) = \sum_{e\colon s(e) = v} \sqrt{\we(e)} \;e.
\]
Then, the condition $\sum_{e\colon s(e) = v} \we(e) = \absv{\twe}$ implies that $B^*B = |\twe|\id$, so $\norm{B} = \sqrt{\absv{\twe}}$.  Furthermore, consider the unitary operator $U$ on $\ell^2 E$ defined by $Ue = \bar{e}$.  Then, $\we(e) \we(\bar{e}) = 1$ implies that $B^* U B$ is equal to the adjacency matrix $A(\Gamma)$ of $\Gamma$.  Thus, we have
\[
\norm{\Gamma} = \norm{A(\Gamma)} \le \norm{B^*} \norm{U} \norm{B} = \absv{\twe}.
\]

Next, suppose that $\we(e) \neq \we(f)$ for some edges satisfying $(s(e), t(e)) = (s(f), t(f))$.  Then, one may modify the above unitary $U$ by setting $U e = \bar{f}$, $U f = \bar{e}$, and otherwise keeping the same definition.  Since
\[
\sqrt{\we(e)} \sqrt{\we(\bar{f})} + \sqrt{\we(f)} \sqrt{\we(\bar{e})} = \frac{\sqrt{\we(e)}}{\sqrt{\we(f)}} + \frac{\sqrt{\we(f)}}{\sqrt{\we(e)}} > 2,
\]
the matrix $B^* U B$ admits an entry strictly larger than that of $A(\Gamma)$ at the $(s(e), t(e))$-th place, and the same ones elsewhere.  Thus we have $\norm{B^* U B} > \norm{A(\Gamma)}$ and $\absv{\twe} > \norm{\Gamma}$.
\end{proof}

\begin{Cor}\label{CorCostEqGrphNorm}
Let $\Gamma$ be a symmetric graph.  Then, any fair and balanced $(-\norm{\Gamma})$-cost on $\Gamma$ comes from a Frobenius--Perron eigenvector at eigenvalue $\|\Gamma\|$ as in the proof of Proposition~\ref{PropExCost}.
\end{Cor}

\begin{proof}
Suppose that we are given a fair and balanced $(-\norm{\Gamma})$-cost on $\Gamma$.  By Proposition~\ref{PropNStepGraph}, we obtain a fair and balanced $(-\norm{\Gamma}^n)$-cost on $\Gamma^{(n)}$ for any $n$.  Since $\Gamma^{(n)}$ has the norm $\norm{\Gamma}^n$, Proposition ~\ref{PropCostVSGrphNorm} implies that the paths of same length, source, and target have the same cost.  Thus, the argument of Remark~\ref{RemPFArg}.(2) also works for $\Gamma$, and we can reconstruct a Frobenius--Perron eigenvector $(c_v)_v$ of $\Gamma$ such that $W(e) = c_{t(e)} / c_{s(e)}$.
\end{proof}

\subsection{Fair and balanced $t$-graphs from finite index subfactors}

A fair and balanced cost structure can also be constructed on the principal graph of a subfactor.  In this section we freely use constructions from subfactor theory, see for example~\cite{Good1} and \cite{Eva1} for details.

Let $N \subset M$ be a II$_1$-subfactor of finite index.  Associated to this inclusion are the $\cC^*$-categories $\bimod{N}{N}$ of $N\mhyph N$-bimodules and $\bimod{M}{N}$ of $M\mhyph N$-bimodules, generated by the object $N \in \tensor[_N]N{_N}$ under the pair of adjoint functors
\begin{align*}
\tensor[_M]M{_N} \otimes - \colon \bimod{N}{N} &\rightarrow \bimod{M}{N},&
\tensor[_N]M{_M} \otimes - \colon \bimod{M}{N} &\rightarrow \bimod{N}{N},
\end{align*}
where the $\otimes$ denote the appropriate Connes fusion products.

Let $\cat{X}$ denote the direct sum $\cC^*$-category generated by a copy of $\bimod{N}{N}$ and a copy of $\bimod{M}{N}$, so that in particular $\Mor(X, Y) = 0$ wherever $X \in \bimod{N}{N}$ and $Y \in \bimod{M}{N}$.  We can define an endofunctor $F$ of $\cat{X}$ by taking $\tensor[_M]M{_N} \otimes -$ on $\bimod{N}{N}$ and $\tensor[_N]M{_M} \otimes -$ on $\bimod{M}{N}$.  Then $F^2$ is given by $\tensor[_N]M{_N} \otimes -$ on $\bimod{N}{N}$, and $\tensor[_M]M{_N} \otimes -$ on $\bimod{M}{N}$.  The latter can be written as $\tensor[_M]{(M_1)}{_M} \otimes -$, with $N\subset M \subset M_1$ the basic construction, since we have a natural $M\mhyph M$-bimodule isomorphism
\begin{equation}\label{EqMoMeqM1}
\tensor{M}{_N} \otimes \ld{N}M \rightarrow M_1, \quad a \otimes b \mapsto [M : N]^{1/2} a e_N b
\end{equation}
which is compatible with the right $M$-valued inner products $\langle a\otimes b, a' \otimes b' \rangle = b^* E_N(a^* a') b'$ on $\tensor{M}{_N} \otimes \ld{N}M $ and $\langle x, y \rangle = E_M(x^* y)$ on $M_1$.

The natural embeddings $N \rightarrow M$ and $M \rightarrow M_1$ induce a natural transformation $R_0$ from $\id_\cat{X}$ to $F^2$, and in particular $F$ may be identified with a self-adjoint object of $\Ef^J$, where $J$ is a parametrization of the irreducible objects in $\cat{X}$.

\begin{Prop}
The transformation $R_0$ satisfies
\begin{align*}
R_0^* R_0 &= \id,&
(R_0^* \otimes \id) (\id \otimes R_0) &= [M : N]^{-1/2} \id.
\end{align*}
\end{Prop}

\begin{proof}
By construction, $R_0^*$ is induced by the conditional expectations $E_N\colon M \rightarrow N$ and $E_M\colon M_1 \rightarrow M$.  This implies the first equality.

Next, let us compute the effect of the operation $(R_0^* \otimes \id) (\id \otimes R_0)$.  If X is an $N$-$N$ bimodule, this is given by $a \otimes x \rightarrow E_M(a \otimes 1) \otimes x$ for $a \in M, x \in X$, where we identify $a \otimes 1$ with an element of $M_1$ as in \eqref{EqMoMeqM1} and apply $E_M$.  Using $E_M(e_N) = [M : N]^{-1}$, we can compute $E_M(a \otimes 1) = [M : N]^{-1/2} a$.  If X is an $M$-$N$ bimodule, $R_0$ can be also written as
\[
x \mapsto [M : N]^{-1/2} \sum_j m_j \otimes m_j^* \otimes x
\]
using a Pimsner--Popa basis $(m_j)_j$ of $M$ over $N$, in the notation of~\cite{Pim1}.  Then, the effect of $(R^* \otimes \id) (\id \otimes R)$ can be expressed as
\[
[M : N]^{-1/2} \sum_j E_N(m_j) m_j^* \otimes x = [M : N]^{-1/2} 1 \otimes x.
\]
This completes the proof.
\end{proof}

It follows that $(F, [M : N]^{1/4} R_0)$ is a $q$-fundamental solution in $\End(\cat{X})$ for the $q$ satisfying $\qn{2} = -[M : N]^{1/2}$.  If we change the embedding $M \rightarrow M_1$ to $x \mapsto -x$ in the definition of $R_0$, we obtain a $q$-fundamental solution with $\qn{2} = [M : N]^{1/2}$.  By construction, the corresponding graph is the principal graph of $N \subset M$.   The `$2$-step construction' of Proposition~\ref{PropNStepGraph} on the even vertices gives a $(1+q^2)$-fundamental solution of Pinzari--Roberts~\cite{Pin4}.

\begin{Rem}
Suppose that the inclusion $N \subset M$ is of finite depth, so that the associated principal graph is a finite graph of norm $[M : N]^{1/2}$.  Corollary~\ref{CorCostEqGrphNorm} implies that the $q$-fundamental solution $(F, R)$ constructed from this inclusion is isomorphic to the one given in the existential part of Proposition~\ref{PropExCost} (this can also be verified by the explicit computation of the Frobenius reciprocity, see~\cite[Theorem~9.71]{Eva1}), and hence does not retain anything from the original subfactor except for its principal graph.  In the infinite depth case, there is room for more structure.  For example, for any irreducible unitary representation $u$ of some compact quantum group, there is a subfactor of index $\dim_q(u)$ whose principal graph is the decomposition graph of $1, u, u \Circt \bar{u}, u \Circt \bar{u} \Circt u, \ldots$~\cite{Ban4,Ued1}.  If $u$ is the fundamental representation of $\SU_q(2)$ itself, we recover the graph $A_\infty$ with the fair and balanced $\aqn{2}$-cost.
\end{Rem}

\subsection{Explicit examples}

\begin{Exa}\label{ExaPodlesSphGraphs}
Consider the graph $A_{\infty,\infty}$ (Figure~\ref{FigAinfinf}), i.e.~the Cayley graph of $(\Z,\{-1,1\})$. For $\absv{q} = 1$, this graph obviously has a unique structure of a fair and balanced $2$-graph, every edge having weight $1$.  When $0<\absv{q}<1$, take $x\in \mathbb{R}$ and define the edge weights $\we_{q, x}$ by
\[
\we_{q,x}(m\rightarrow m+1)=\absv{\frac{q^{x+m+1}+q^{-x-m-1}}{q^{x+m}+q^{-x-m}}},\quad \we_{q, x}(m+1\rightarrow m) = \absv{\frac{q^{x+m}+q^{-x-m}}{q^{x+m+1}+q^{-x-m-1}}},
\]
and put $\we_{q, \infty}(m \rightarrow m + 1) = q^{-1}$ and $\we_{q, \infty}(m  + 1 \rightarrow m) = q$.  Then the $(A_{\infty,\infty},\we_{q,x})$ are all fair and balanced $\qn{2}$-graphs based on $A_{\infty,\infty}$.

To see this, choose a fair and balanced $\qn{2}$-cost, and put $a = \we(0\rightarrow 1)>0$. Then we have
\begin{align}\label{EqnIneqa}
a &\leq \aqn{2}, &
a^{-1} &\leq \aqn{2}.
\end{align}
Let us write $\we(1\rightarrow 2) = b$ and $\we(-1\rightarrow 0) = c$. By the constant source weight condition, $b+a^{-1} = \aqn{2}$ and $c^{-1}+a = \aqn{2}$. Since \eqref{EqnIneqa} also holds for $b$ and $c$, we find
\begin{align*}
a &\leq \frac{\aqn{3}}{\aqn{2}},&
a^{-1} &\leq \frac{\aqn{3}}{\aqn{2}}.
\end{align*}
By induction, we deduce that
\begin{align*}
a &\leq \frac{\aqn{n+1}}{\aqn{n}},&
a^{-1}&\leq \frac{\aqn{n+1}}{\aqn{n}},
\end{align*}
and hence, taking a limit $n\rightarrow \infty$, that $q\leq a \leq q^{-1}$. But each such $a$ can be written as $\frac{q^{x+1}+q^{-x-1}}{q^{x}+q^{-x}}$ for a unique $x \in \mathbb{R}\cup \{\pm{\infty}\}$. Since the $\we(m\rightarrow m+1)$ are obviously determined by the value of $a$, and since the weights $\we_{q,x}$ as above are easily seen to give a fair and balanced $\qn{2}$-graph, our claim is shown.

Since the flip $m \mapsto -m$ exchanges parameters $x$ and $-x$, and since the translation $m\mapsto m\pm 1$ exchanges parameters $x$ and $x\mp 1$, we may restrict to the case of $0 \le x<1$.  In this range, the weighted graphs are mutually non-isomorphic and exhaust all the possibilities of such weighted graphs.  The vertices $m\in \Z$ of $(A_{\infty,\infty},\we_{q,x})$ correspond to the family of Podle\'{s} spheres $S_{qc(x+m)}^2$ for $\SU_q(2)$ for a particular parametrization $x\rightarrow c(x)$, see Example \ref{ExGenPod}.
\begin{figure}[p]
\centerline{
\xymatrix@C=8em{ \cdots \underset{m-1}{\bullet} \ar @/^/[r]^{\absv{\frac{q^{x+m}+q^{-x-m}}{q^{x+m-1}+q^{-x-m+1}}}}  & \ar @/^/[l]^{\absv{\frac{q^{x+m-1}+q^{-x-m+1}}{q^{x+m}+q^{-x-m}}}} \underset{m}{\bullet} \ar @/^/[r]^{\absv{\frac{q^{x+m+1}+q^{-x-m-1}}{q^{x+m}+q^{-x-m}}}} & \underset{m+1}{\bullet} \ar @/^/[l]^{\absv{\frac{q^{x+m}+q^{-x-m}}{q^{x+m+1}+q^{-x-m-1}}}}  \cdots}
}
\caption{$A_{\infty,\infty}$}
\label{FigAinfinf}
\end{figure}
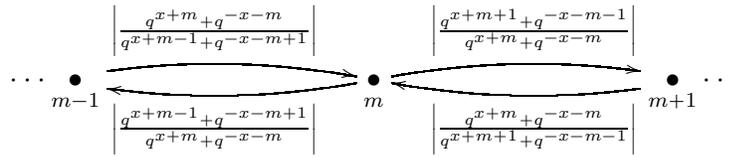
\end{Exa}

\begin{Exa}\label{ExaQRProjSp}
For any $0 < \absv{q} \le 1$, the graph $D_\infty^{*}$ (Figure~\ref{FigDinf}) admits a unique structure of a fair and balanced $\qn{2}$-graph.  The quantum homogeneous space corresponding to either of the endpoints is the quantum real projective plane algebra $\CoL{\R P^2_q}$ (\cite{Haj1,Haj2}), see also Example~\ref{ExaQRProjEmb}.
\end{Exa}
\begin{figure}[p]
\begin{minipage}[b]{0.45\linewidth}
\centerline{
\xymatrix@R=0.5em@C=4em{
\overset{*}{\bullet} \ar@/^/[dr]^{\aqn{2}} \\
& \underset{1}{\bullet} \ar@/^/[ul]^(0.7){} \ar@/_/[dl] \ar@/^/[r]^{\absv{\frac{q^2 + q^{-2}}{q + q^{-1}}}} & \underset{2}{\bullet} \ar@/^/[l]^{\absv{\frac{q + q^{-1}}{q^2 + q^{-2}}}} \ar@/^/[r]^{\absv{\frac{q^3+q^{-3}}{q^2 + q^{-2}}}} & \underset{3}{\bullet} \ar@/^/[l]^{\absv{\frac{q^2 + q^{-2}}{q^3+q^{-3}}}} \cdots\\
\underset{\tilde{*}}{\bullet} \ar@/_/[ur]_{\aqn{2}}
}
}
\caption{$D_\infty^{*}$}
\label{FigDinf}
\end{minipage}
\hspace{0.5cm}
\begin{minipage}[b]{0.45\linewidth}
\centerline{
\xymatrix@=0.5em{
& & & \ar @/^/[dlll]^{\absv{q}} \\
\bullet \ar @/^/[urrr]^{\absv{q}^{-1}} \ar @/^/[ddd]^{\absv{q}} \\
\\
\\
\bullet \ar @/^/[uuu]^{\absv{q}^{-1}}  \ar @/^/[drrr]^{\absv{q}} \\
& & &\ar @/^/[ulll]^{\absv{q}^{-1}} \ar@{--}@(r,r)[uuuuu]
}
}
\caption{$A_n^{(1)}$ ($n \ge 1$, $n+1$ vertices)}
\label{FigAn}
\end{minipage}
\end{figure}

\begin{Exa}\label{ExaCyclicSubgrp}
Consider the graphs $A_n^{(1)}$ for $1 \le n < \infty$ (Figure~\ref{FigAn}). Then one checks that the only possible fair and balanced $\aqn{2}$-cost on it is obtained by giving the weight $\absv{q}$ to (say) the counter-clockwise edges. This corresponds to the subgroup $\Zz_{n+1}\subseteq U(1) \subseteq \SU_q(2)$.
\end{Exa}

\begin{Exa} Consider the graph $E_6^{(1)}$ (Figure~\ref{FigE6}). Then one can check that this graph has a structure of a fair and balanced $\aqn{2}$-graph only when $\absv{q}=1$.  The same goes for the graphs $E_7^{(1)}$, $E_8^{(1)}$, and $D_n^{(1)}$ for $4 \le n$. These correspond to the finite subgroups of $\SU_{\pm 1}(2)$ which are central extensions by $\Zz_2$ of $A_4$, $S_4$, $A_5$ and the dihedral groups of order $2 n$.
\end{Exa}

\begin{Exa}\label{Exaqminus1AmDn}
For $q = -1$, we find a unique structure of compatible fair and balanced $(-2)$-graph for each of the types $A_m'$ ($2 \le m \le \infty$, Figure~\ref{FigAm'}), $D_m'$ ($3 \le m < \infty$, Figure~\ref{FigDm'}).

\begin{figure}[p]
\begin{minipage}[b]{0.45\textwidth}
\centerline{
\xymatrix@=3em{
&&\bullet \ar @/^/[d]^{^2} &&
\\ && \bullet \ar @/^/[d]^{^{3/2}} \ar @/^/[u]^{^{1/2}} &&
\\ \bullet  \ar @/^/[r]^{_2}  & \ar @/^/[l]^{^{1/2}} \ar @/^/[r]^{_{3/2}} \bullet & \ar @/^/[l]^{^{2/3}} \ar @/^/[r]^{_{2/3}} \ar @/^/[u]^{^{2/3}} \bullet & \ar @/^/[l]^{^{3/2}} \ar @/^/[r]^{_{1/2}} \bullet& \ar @/^/[l]^{^2} \bullet
}
}
\caption{$E_6^{(1)}$}
\label{FigE6}
\end{minipage}
\hspace{0.5cm}
\begin{minipage}[b]{0.45\textwidth}
\centerline{
\xymatrix{
\bullet \ar@/^/[dr]^{2} \\
& \bullet \ar@/^/[ul]^{1/2} \ar@/_/[dl]_{1/2} \ar@/^/[r]^{1} &  \bullet \ar@/^/[l]^{1} \ar@/^/[r]^{1} & \cdots \ar@/^/[l]^{1} \ar@/^/[r]^{1} & \bullet \ar@/^/[l]^{1} \ar@(ur,dr)[0,0]^{1} \\
\bullet \ar@/_/[ur]_{2}
}
}
\caption{$D_m'$ ($3 \le m$, $m$ vertices)}
\label{FigDm'}
\end{minipage}
\end{figure}
\end{Exa}

\begin{Exa}\label{ExaGraphAinfty'}
For a generic negative $q$, the only new example is the graph $A_\infty'$ (Figure~\ref{FigAinf'}) whose vertices are labeled by nonnegative integers.  One may give a compatible fair and balanced $\qn{2}$-graph structure as follows.  Notice that the weights on rightward edges are monotonically increasing when $-1 < q$.  This also gives a coideal, see Example~\ref{ExaCoidTypeAinf'} and~\cite{Tom2}.
\begin{figure}[p]
\begin{minipage}[b]{0.45\textwidth}
\centerline{
\xymatrix{
\bullet \ar@(dl,ul)[0,0]^{1} \ar@/^/[r]^{1} & \cdots \ar@/^/[r]^{1} \ar@/^/[l]^{1} & \bullet \ar@(ur,dr)[0,0]^{1} \ar@/^/[l]^{1}
}
}
\caption{$A_m'$ ($2 \le m$, $m$ vertices)}
\label{FigAm'}
\end{minipage}
\hspace{0.5cm}
\begin{minipage}[b]{0.45\textwidth}
\centerline{
\xymatrix@C=5em{
\underset{0}{\bullet} \ar@(dl,ul)[0,0]^{1} \ar@/^/[r]^{\atn{3}}  & \underset{1}{\bullet} \ar@/^/[l]^{\frac{1}{\atn{3}}} \ar@/^/[r]^{\frac{\atn{5}}{\atn{3}}}& \cdots \ar@/^/[l]^{\frac{\atn{3}}{\atn{5}}}
}
}
\caption{$A_\infty'$ $(t = i|q|^{1/2})$}
\label{FigAinf'}
\end{minipage}
\end{figure}
\end{Exa}

\begin{Exa} Consider the graph $\xymatrix{\bullet \ar @<0.5ex>[r] & \ar @<0.5ex>[l] \bullet}$. Then this does not admit the structure of a fair and balanced $\qn{2}$-graph for any $q$ with $0<|q|\leq 1$.  Similarly, when $0 < \left| q \right| < 1$, the graphs $D_n$ for odd $n$, and the graphs $A_m'$ for $3 \le m < \infty$, in the notation of \cite[Appendix]{Tom1}, do not admit fair and balanced $\qn{2}$-costs for $0<|q|\leq 1$.  Hence we obtain that there is no quantum homogeneous space corresponding to these graphs, refining the result of Tomatsu stating the non-existence of such algebras among the coideals of $\CoL{\SU_q(2})$.\footnote{This (see also Example~\ref{ExaQ-1Cids}) corrects the error in~\cite[Theorem~7.1]{Tom1}, which claimed the existence of a coideal of type $D_1$ and the nonexistence of the types $A_m'$ for any negative $q$ due to a circular argument in its proof; Tomatsu has notified us of a correct proof based on the `generators and relations' approach~\cite
 {Tom2}.  Our result agrees with his corrections.}
\end{Exa}

\section{Classification of $SU_q(2)$-homogeneous spaces for $|q|\approx1$}

Using the results of the previous two sections, let us give now a direct combinatorial proof of the classification of quantum homogeneous spaces for $\SU(2)$, as first obtained by Wassermann in \cite{Was1} by using `half' of the categorical structure and more computational techniques.  We shall say that a graph is of extended ADE type if it is a point and a double loop, or one of $A_n^{(1)}$ for $1 \ge n$, $D_n^{(1)}$ for $4 \le n$, $E_n^{(1)}$ for $n = 6, 7, 8$, $A_{\infty, \infty}$, $D_\infty^{*}$, or $A_\infty$, that is, one of the graphs appearing in~\cite{Was1}.

\begin{Prop}\label{PropW2}
Up to equivariant Morita equivalence, the quantum homogeneous spaces over $\SU(2)$ are all of the form $H\backslash\SU(2)$ for some closed subgroup $H$ of $\SU(2)$.
\end{Prop}

\begin{proof}
We first claim that the associated graph must be of the extended ADE type (cf.~\cite[Theorem~1]{Was1}).  Indeed, by Proposition~\ref{PropCostVSGrphNorm}, the possible graphs must have norm at most $2$.  Moreover, the ones with norm smaller than $2$ are trees, hence cannot happen by Remark~\ref{RemPFArg}.2.  Since the involution has to be fixed point free, we obtain that the graph is of extended ADE type.

Now, Corollary~\ref{CorCostEqGrphNorm} implies that the fair and balanced cost is uniquely determined by the graph structure.  Hence the only classification parameter is the graph alone.  As is well known, all these graphs are realized by closed subgroups of $\SU(2)$.
\end{proof}

We obtain an analogous classification for the case of $q = -1$ by the same strategy as above.

\begin{Prop}\label{PropWOqmin1}
The quantum homogeneous spaces over $\SU_{-1}(2)$ are all strongly Morita equivalent to a coideal of $C(\SU_{-1}(2))$.
\end{Prop}

\begin{proof}
The proof is completely analogous to that of the previous proposition.  This time, there is no condition on the involution, hence the graphs of types $A_m'$ ($m = 2, 3, \ldots, \infty$) and $D_m'$ ($m = 3, 4, \ldots$) are also allowed.  Again, any of these graphs are known to correspond to a coideal (see~\cite{Tom1,Tom2} or Section~\ref{SecEqMor} of this paper).
\end{proof}

Further expanding on this approach, we can go a little further and also list explicitly the quantum homogeneous spaces over $\SU_q(2)$ when $\absv{q}$ is very close to $1$.

\begin{Theorem}\label{TheoBound}
There exists a real number $0 < q_0 < 1$ such that for any $q_0 < q < 1$, all quantum homogeneous spaces over $\SU_q(2)$ arise as coideals, and are thus isomorphic to either a single point, the quantum real projective plane $\R P^2_q$, one of the Podle\'{s} spheres, or $\Zz_n \backslash \SU_q(2)$ for some $n\in \N_{>0}$.

For the same constant $q_0$, for any $-1 < q < -q_0$, again all quantum homogeneous spaces over $\SU_q(2)$ arise as coideals.  Other than the ones listed above, we have the possibility of type $A_\infty'$.
\end{Theorem}

\begin{proof}
It is well known that there is a spectral gap between $2$ and the next number $t_0$ which can arise as the norm of an adjacency matrix (\cite[Theorem I.1.2.]{Good1}, the case of infinite graphs being easily incorporated as unions of finite graphs).  Then, with $q_0 = t_0/2$, one has $\aqn{2} < t_0$ whenever $q_0 < \absv{q} < 1$.  By Proposition~\ref{PropCostVSGrphNorm}, any graph associated with an ergodic action of $\SU_q(2)$ for such $q$ must have the graph norm $2$, i.e. must be of extended ADE type.  Given any fair and balanced $\qn{2}$-cost on such a graph, one can choose an appropriate vertex and solve the equation for an equivariant homomorphism into $C(\SU_q(2))$, which gives a realization as an coideal.
\end{proof}

\section{Generators and relations}

We continue to write $u = u_{1/2}$ for the fundamental representation of $\SU_q(2)$.  The standard basis of $\Hsp_{1/2}$ is denoted by $(e_1, e_2)$. Let $(\Gamma,\we)$ be a connected fair and balanced $q+q^{-1}$-graph, with associated $q$-fundamental solution $(\Hsp,\Rmat)$ in $\Ef^J$ where $J= \Gamma^{(0)}$. Let $\mathcal{F}$ be the associated strict tensor functor from $\Rep(\SU_q(2))$ into $\cat{E}_f^J$. Let $\cat{D}^J$ be a $\Rep(\SU_q(2))$-module C$^*$-category associated with $\mathcal{F}$, which we may also assume to be strict. We may use $J$ to label a maximal set $\{x_v\mid v\in J\}$ of mutually non-isomorphic irreducible objects in $\cat{D}$. We recall that we can then identify $\mathcal{F}(u_a)_{v w}$ with $\Mor(x_v,u_a\otimes x_w)$ for $a\in \frac{1}{2}\mathbb{N}$, the latter having the natural Hilbert space structure $\langle f,g\rangle = f^*g$.

In \cite[Section~5]
{DCY2}, we associated with each $v$, $w$ a certain vector space \[\aA_v^w = \oplus_{a\in \frac{1}{2}\mathbb{N}} \Mor(u_a\otimes x_w,x_v)\otimes \mathscr{H}_a.\] These $\aA_v^w$ could be seen as pre-Hilbert C$^*$-equivalence bimodules between $^*$-algebras $\aA_{v}^v$ and $\aA_{w}^w$, equipped with compatible coactions. We recall that the completion of some $\aA_v^v$ is precisely the ergodic action of $\SU_q(2)$ associated with the vertex $v$. In general it is difficult to give a concise description by generators and relations of an individual algebra $\aA_v^v$ in terms of its associated graph (but see \cite{Pin4} for at least \emph{a} description). However, the `linking $^*$-algebra' $\aA=\oplus_{v,w} \mathscr{A}_{v}^w$ has a much nicer presentation. Since $\aA$ is Morita equivalent to any of the algebras $\aA_v^v$ (by means of the $\aA_{vv}$-$\aA$-equivalence bimodule $\oplus_w\mathscr{A}_{v w})$, it has the same $^*$-representation theory as any of the $\aA_{vv}$, and shares many properties  with them, such as the type of associated von Neumann algebras.

We fix an orthonormal basis of each (non-zero) $\mathscr{H}_{v w}=\mathcal{F}(u)_{v w}=\Mor(x_v,u\otimes x_w)$ once and for all, and denote their disjoint union over all $v,w$ as $(f_i)_{i\in L}$ for some index set $L$.  In order to specify which component $f_i$ is in, we define maps $s$ and $t$ from $L$ to $J$ by $(s(i),t(i))=(v,w)$ if $f_i \in \mathcal{F}_{v w}$.

\begin{Def}
For $i\in L$ and $j\in \{1,2\}$ with $(s(i),t(i))= (v,w)$, we denote by $z_{i j}$ the element \[ z_{i j} = f_i^* \otimes e_j \in \Mor(u\otimes x_w,x_v)\otimes \Hsp_u  \subseteq \aA_{v}^{w} \subseteq \aA.\]  The unit of $\aA_v^v$, considered as a projection in $\aA$, is denoted by $\delta_v$.
\end{Def}

The coaction of $\PW(\SU_q(2))$ on the various $\aA_v^w$ as in \cite[Definition~5.2]
{DCY2} combine to a global coaction $\alpha$ of $\PW(\SU_q(2))$ on $\aA$, by \cite[Proposition~5.8 
and Lemma~5.13]
{DCY2}. The formula for the coaction on the generators $z_{i j}$ can be immediately deduced from the definition of this coaction.

\begin{Lem}
The coaction $\alpha\colon \aA \rightarrow \aA \otimes \PW(\SU_q(2))$ is determined by the formulas
\begin{align*}
\alpha(\delta_v) &= \delta_v \otimes 1,&
\alpha(z_{i j}) &= \sum_{k=1}^2 z_{i k}\otimes u_{k j}.
\end{align*}
\end{Lem}

We next deduce certain relations among the elements $\delta_v$ and $z_{i j}$. First, the following equalities are obvious by construction.
\begin{equation}\label{Eq1}
\delta_v z_{i j} \delta_w = \delta_{v,s(i)}\delta_{w,t(i)}z_{i j}.
\end{equation}

\begin{Lem}\label{LemEqvAuFCorepOrth}
For any $w \in J$, one has the following relation inside $\aA$:
\begin{equation}\label{Eq2}
\sum_{i\in L, t(i)=w} z_{i j}^*  z_{i k} = \delta_{j, k}\delta_w.
\end{equation}
\end{Lem}

Note that the sum on the left hand side is finite, as $\Hsp_{v w}$ is non-zero for only a finite number of $v$ when $w$ is fixed.

\begin{proof} By definition of the $^*$-operation \cite[Definition~5.10]
{DCY2}, we have that $z_{ij}^*$ has only non-zero component at $a=1/2$, where it equals the element $\lbrack (R_u^*\otimes \id_{t(i)})(\id_u\otimes f_i)\rbrack \otimes (e_j^*\otimes \id_u)(\bar{R}_u(1))$. By definition of the multiplication \cite[Definition~5.4]
{DCY2}, we have then that $\sum_{i, t(i)=w} z_{ij}^*z_{ij}$ corresponds to the element \[ \sum_{i,t(i)=w} \lbrack (R_u^*\otimes \id_{w})(\id_u\otimes f_if_i^*)((((e_j^*\otimes \id_u)\bar{R}_u(1))\otimes e_j)^c\otimes \id_w)\rbrack \otimes (((e_j^*\otimes \id_u)\bar{R}_u(1))\otimes e_j)_c\] (see also the proof of~\cite[Theorem~5.16]
{DCY2}). But since the $f_i$ are orthonormal, $\sum_{i,t(i)=w} f_i f_i^*$ is equal to $\id \in \End(u \otimes x_w)$. Then, by looking at the morphism part, we see that the remaining expression only has a non-zero component at $c=0$. Thus the above can be simplified to $\id_w \otimes (e_j^*\otimes R^*) (\bar{R}(1)\otimes e_k) = \id_w \otimes \langle e_j, e_k \rangle$, proving the assertion.
\end{proof}

\begin{Lem}\label{LemEqvCopiesHuOrth}
For any $i, k \in L$, one has
\begin{equation}\label{Eq2'}
z_{i 1} z_{k 1}^* + z_{i 2} z_{k 2}^* = \delta_{i,k} \delta_{s(i)}.
\end{equation}
\end{Lem}

\begin{proof}
Applying the coaction to the above formula, the unitarity of $u$ implies that the left hand side represents an invariant element. Since it lies in $\aA_{s(i)}^{s(i)}$, it must be a scalar multiple of $\id_{s(i)}$. Applying $E_{\G}$ to it (cf.~\cite[Lemma~5.15]
{DCY2}) and using the conjugate equations, we obtain that this scalar must be $\sum_j \frac{1}{\dim_q(u)}\langle \bar{R}_u(1),e_j\otimes ((e_j^*\otimes \id_u)\bar{R}_u(1))\rangle f_i^*f_k$, which by the orthogonality of the $f_i$ reduces to $\delta_{i,k}$.
\end{proof}

The final relation expresses the adjoints $z_{i j}^*$ as a linear combination of elements of the form $z_{k l}$. It is here that the weight $\we$ on the graph will be used. Recall from \eqref{EqRandJ} the anti-linear maps $\Jop_{v w}$ associated with our $q$-fundamental solution $(\Hsp,\Rmat)$. Consider the matrix $E^{(v w)}_{i j} = -\sgn(q) \langle \Jop_{v w} f_j, f_i \rangle$, where $s(j) = t(i) = v$ and $t(j)=s(i)=w$. Then the identity $\Jop_{v w} \Jop_{w v} = -\sgn(q)$ can be written as
\[
\sum_{\substack{k \in L \\(s(k),t(k))=(v,w)}}\overline{E_{i k}^{(v w)}} E_{kj}^{(w v)}  = -\sgn(q) \delta_{i, j}.
\]

Recall also the matrix $F$ introduced after Definition \ref{DefSUq2}.

\begin{Lem}\label{LemFormStar}
For any $i \in L$ and $j \in \{1,2\}$ with $(s(i),t(i))=(v,w)$, we have \begin{equation}
\label{Eq3} z_{i j}^* =\sum_{\substack{k \in L \\(s(k),t(k))=(w,v)}} E^{(w,v)}_{i k} (F_{1 j}z_{k 1} + F_{2 j}z_{k 2}),
\end{equation}
or more succinctly
\[
\overline{z^{(v,w)}} = E^{(w,v)} z^{(w,v)} F,
\]
where $z^{(v,w)}$ is the matrix of $z_{i j}$ with $(s(i),t(i))= (v,w)$.
\end{Lem}

\begin{proof} As in the proof of Lemma \ref{LemEqvAuFCorepOrth}, we have that $z_{ij}^*$ has only a non-zero component at $a=1/2$, where it equals the element $\lbrack (R_u^*\otimes \id_{w})(\id_u\otimes f_i)\rbrack \otimes (e_j^*\otimes \id_u)(\bar{R}_u(1))$. Now $(R_u\otimes \id_w) \in \Mor(x_w,u\otimes u\otimes x_w)$, and the latter space can be identified with $\mathcal{F}(u\otimes u)_{ww} = \oplus_v \mathcal{F}(u)_{w v}\otimes \mathcal{F}(u)_{vw}$, by strict tensoriality of $\mathcal{F}$. Under this correspondence, the element $(R_u\otimes \id_w)$ coincides, by construction, with the element $\oplus_v \mathcal{R}_{w v}(1)$. But the latter can be written \[\oplus_v \mathcal{R}_{w v}(1)=\sum_{v,k,l} \langle f_l,\Jop_{w v}f_k\rangle f_k\otimes f_l.\] Identifying again with $\Mor(x_w,u\otimes u\otimes x_w)$, we obtain the equality \[ R_u\otimes \id_w = \sum_{v,k,l} \langle f_l,\Jop_{w v}f_k\rangle (\id_u\otimes f_l)f_k.\] Plugging this into the expression for $z_{ij}^*$ above, and using orthogonality of the $f_i$, we obtain \[z_{ij}^* = \sum_k \langle \Jop_{w v}f_k,f_i\rangle f_k^* \otimes (e_j^*\otimes \id_u)\bar{R}_u(1).\] Since $-\sgn(q)(e_j^*\otimes \id_u)\bar{R}_u(1) = F_{1j}e_1+F_{2j}e_2$, this is precisely the formula for $z_{ij}^*$ as stated in the lemma.
\end{proof}

We now aim to show that these relations are the \emph{universal} ones by which one can define $\aA$. We write $C_c(J)$ for the algebra of finitely supported functions on $J$, and we denote $\delta_v \in C_c(J)$ for the Dirac function at $v \in J$. We identify $C_c(J)$ with its copy inside $\aA$.

\begin{Theorem}
Let $\mathcal{A}$ be the (not necessarily unital) $^*$-algebra generated by a copy $C_c(J)$ and elements $Z_{i j}$ for $i\in L$, $j\in \{1,2\}$ satisfying the relations \eqref{Eq1}, \eqref{Eq2}, \eqref{Eq2'}, and \eqref{Eq3} in place of $z_{i j}$. Then the natural $^*$-homomorphism
\[
\pi\colon \mathcal{A} \rightarrow \aA, \quad
\begin{cases}
\delta_v \mapsto \delta_v,\\
Z_{i j} \mapsto z_{i j},
\end{cases}
\]
is an isomorphism.
\end{Theorem}

\begin{proof} Let us first prove surjectivity of $\pi$. That is, we want to show that the $z_{i j}$ generate $\aA$.  If this is not the case, there would be a nonzero morphism $f \in \Mor(u_a \otimes x_v, x_w)$, for some $a \in \half\N$ and $v, w \in J$, which is orthogonal to the morphisms in $\Mor(u^{\smCirct k} \otimes x_v, x_w)$ for any $k \in \N$.  Since any $u_a$ is contained in some $u^{\smCirct k}$, this is impossible.

To prove now the bijectivity of $\pi$, we construct an explicit formula for its inverse. To obtain it, we modify a technique which is very useful in the manipulation of Galois objects (cf.~\cite{Bic2}, see also the more direct proof in \cite{BDV1}).

We first remark that there exists a coaction of $\PW(\SU_q(2))$ on $\mathcal{A}$, which we will also denote as $\alpha$, uniquely determined by the formula \[\alpha(Z_{i j}) = \sum_{k=1}^2 Z_{i k} \otimes u_{k j}\] That $\alpha$ is well-defined follows from the form of this formula (for \eqref{Eq1}) and the fact that $(u_{i j})_{i, j}$ is unitary (for \eqref{Eq2} and \eqref{Eq2'}).  Moreover by~\eqref{Eq3}, $\alpha$ is a $^*$-homomorphism.  From the form of $\alpha$, it also satisfies the coaction property, and hence defines a coaction of $\PW(\SU_q(2))$. Note that $\pi$ is then an equivariant map. In the following, we will use the shorthand notation $\alpha(x) = x_{(0)} \otimes x_{(1)}$ both for $x\in \aA$ and $x\in \mathcal{A}$.

Denote now by $\mathcal{B}$ the (not necessarily unital) $^*$-algebra generated by a copy of $C_c(J)$ and elements $Y_{i j}$ for $i \in \{1,2\}$ and $j \in L$, subject to $\delta_r Y_{i j} \delta_s = \delta_{s(i),v}\delta_{t(i),w}Y_{i j}$ and the relations
\begin{align*} 
 \sum_{j=1}^2 Y_{i j}Y_{k j}^* &= \delta_{i,k}\delta_{s(i)},&&
 \sum_{\substack{j\in L\\t(j)=w}} Y_{j i}^*Y_{j k} = \delta_{i,k} \delta_{w},&&
 \overline{Y^{(v,w)}} = F^{-1} Y^{(w,v)} (E^{(w,v)})^{-1},
\end{align*}
which are simply the relations defining the opposite algebra of $\mathcal{A}$.

By the defining relations for $\PW(\SU_q(2))$, the formula
\[
\beta\colon \PW(\SU_q(2)) \rightarrow \prod_{w \in J} \bigoplus_{v\in J} \Bigl(\mathcal{B}_{v w} \otimes \mathcal{A}_{v w} \Bigr),\quad u_{i j}\mapsto \Bigl(\underset{t(k)=w}{\sum_{k\in L}} Y_{i k}\otimes Z_{k j}\Bigr)_{w}
\]
defines a unital $^*$-homomorphism, where we put $\mathcal{B}_{v w} = \delta_v \mathcal{B} \delta_w$, and where the codomain is considered inside a completed tensor product of the algebras $\mathcal{B}$ and $\mathcal{A}$, or alternatively as left multiplication operators on $\mathcal{B}\otimes \mathcal{A}$. Let us write the components of $\beta(x)$ as $\beta(x)_{vw} \in \mathcal{B}_{v w} \otimes \mathcal{A}_{v w}$. We will employ the shorthand notation $\beta(x)_{vw} = x_{[1,vw]}\otimes x_{[2,vw]}$.

As a final ingredient, define an anti-isomorphism
\[
S\colon \mathcal{B}\rightarrow \mathcal{A}, \quad \begin{cases}
\delta_v &\mapsto \delta_v,\\
Y_{i j} &\mapsto Z_{j i}^*,
\end{cases}
\]
which is well-defined by the symmetry of our relations. Then $S$ restricts to maps $S_{vw}\colon \mathcal{B}_{vw}\rightarrow \mathcal{A}_{wv}$.

Write $E_{\G}$ for the map $\aA\rightarrow C_c(J)$ sending $x$ to $(\id \otimes \varphi_{\G})\alpha(x)$. Define then \[\phi_{vw}\colon \aA_{vw}\rightarrow \mathcal{A}_{vw}, \qquad x \rightarrow E_{\G}(x_{(0)}\pi(S_{vw}(x_{(1)[1vw]})))x_{(1)[2vw]}.\] We claim that $\phi = \oplus_{v,w}\phi_{vw}$ is an inverse for $\pi$.

To see this, pick $y\in \mathcal{A}_{vw}$. Then $\phi_{vw}(\pi(y)) = E_{\G}(\pi(y_{(0)}S_{vw}(y_{(1)[1vw]})))y_{(1)[2vw]}$ by equivariance of $\pi$. We argue that in fact \begin{equation}\label{eqnGalCo} y_{(0)}S_{vw}(y_{(1)[1vw]})\otimes y_{(1)[2vw]} = \delta_{v}\otimes y.\end{equation} Indeed, using the multiplicativity of $\alpha$ and $\beta$ and the anti-multiplicativity of $S$, we have that equation \eqref{eqnGalCo} holds for a product if it holds for the separate factors. Hence \eqref{eqnGalCo} only has to be verified on the generators $Z_{ij}$, $Z_{ij}^*$ and $\delta_i$, which follows from an easy computation using the defining relations. It follows that $\phi(\pi(y)) = y$, and that $\phi$ is the inverse of $\pi$.

\end{proof}

\begin{Exa}\label{ExGenPod} Consider the weighted graph $(A_{\infty,\infty},\we_{q,x})$ as in Example \ref{ExaPodlesSphGraphs}. In this case, the generators $z$ can be labeled as $z_i^{(m,m\pm 1)}$, since $\Hsp_{mn}$ is zero or one-dimensional, the latter case only arising if $|m-n|=1$.

From \eqref{LemEqvCopiesHuOrth}, we get that \begin{equation}\label{EqUnita} z_1^{(1 0)}z_1^{(1 0)*}+ z_2^{(1 0)}z_2^{(1 0)*} = 1 = z_1^{(0 1)}z_1^{(0 1)*}+ z_2^{(0 1)}z_2^{(0 1)*}.\end{equation} Now we may choose $E^{(1 0)} = \we_{q,x}(1\rightarrow 0)^{1/2}$ (and then $E^{(0 1)} = -\sgn(q)\we_{q,x}(1\rightarrow 0)^{-1/2}$), so by Lemma \ref{LemFormStar} we know
\begin{equation}\label{EqAdjo}
\begin{cases}
z_1^{(0 1)*} = -\sgn(q)|q|^{-1/2}\we_{q,x}(1\rightarrow 0)^{1/2}z_2^{(1 0)},  \\
z_2^{(0 1)^*} = |q|^{1/2}\we_{q,x}(1\rightarrow 0)^{1/2}z_1^{(1 0)}.\end{cases}
\end{equation}
Hence, abbreviating $z_i = z^{(1 0)}_i$ and $\we_{q,x}(1\rightarrow 0) = \we_{x}$, we can use \eqref{EqAdjo} to rewrite \eqref{EqUnita} as
\begin{equation}\label{EqSimpl}
\begin{cases}
z_1z_1^*+ z_2z_2^* = 1\\
|q|z_1^*z_1+ |q|^{-1}z_2^*z_2 = \we_{x}^{-1}.
\end{cases}
\end{equation}
Let us write
\begin{equation*}
\begin{cases}
X = (|q|^{-x}+|q|^{x})z_2^*z_1,\\
Z = (|q|^{-x+1}+|q|^{x+1})(z_1^*z_1- \frac{|q|^{x}}{|q|^{-x}+|q|^{x}}),\\
Y =  (|q|^{-x}+|q|^{x})z_1^*z_2.
\end{cases}
\end{equation*}
Then some easy but slightly tedious computations using only \eqref{EqSimpl} show that the $X$, $Y$, $Z$ satisfy the relations of the Podle\'{s} sphere $S_{qc(x)}^2$ with $c(x) = (|q|^{x+1}-|q|^{-x-1})^{-2}$, with the parametrization as in \cite{Pod1}, namely $X^*=Y$, $Z^*=Z$, $XZ=q^{2}ZX$, $YZ=q^{-2}ZY$, and
\[
\begin{cases}
YX = (1-|q|^{-x-2}Z)(1+|q|^{x}Z)\\
XY = (1-|q|^{-x}Z)(1+|q|^{x+2}Z).
\end{cases}
\]
Another calculation shows that this gives an equivariant map from $\PW(S_{qc(x)}^2)$ to $\aA_0^0$.

Now, since $\PW(S_{qc(x)}^2)$ has a faithful positive invariant state, and since the action on it is ergodic, the map from $\PW(S_{qc(x)}^2)$ into $\aA_0^0$ is injective. But since the multiplicity graph of $\aA_0^0$ is the same as the one of $\PW(S_{qc(x)}^2)$, they must have the same spectral decomposition. It follows that the above map must actually be an isomorphism.
\end{Exa}

\section{Equivariant morphisms}\label{SecEqMor}

The result of~\cite[Section~7]
{DCY2} can be specialized to the $\SU_q(2)$ case in a form only involving the $q$-fundamental solutions.  We continue to abbreviate $u = u_{1/2}$. If $\X$ is a quantum homogeneous space for $\SU_q(2)$, we write $\cat{D}_\X$ for the $\Rep(\SU_q(2))$-module C$^*$-category of finite equivariant Hilbert C$^*$-modules, and we write the module explicitly by $M$ as in \cite[Definition~2.14]
{DCY2} for emphasis. We denote by $\cat{E}_\X$ the tensor C$^*$-category of endofunctors of $\X$, and $\Rmat^\X$ for the associated $q$-fundamental solution.

Let now $\X$ and $\Y$ be quantum homogeneous spaces over $\SU_q(2)$. If there is an $\SU_q(2)$-morphism from $\Y$ to $\X$ represented by an equivariant homomorphism $f\colon C(\X) \rightarrow C(\Y)$, the induced $\Rep(\SU_q(2))$-homomorphism $(\theta_\#, \psi)$ from $\cat{D}_\X$ to $\cat{D}_\Y$ must make the diagram
\[
\xymatrix{
\theta_\# x \ar[d]_{\theta_\#(M(\Rmat^\X, x))} \ar[r]^-{M(\Rmat^\Y, \theta_\# x)} & u \otimes u \otimes \theta_\# x \ar[dl]^{\psi_{u \otimes u, x}}\\
\theta_\# (u \otimes u \otimes x)
}
\]
commutative for arbitrary object $x \in \cat{D}_\X$.  Let $J$ (resp. $J'$) be an index set of irreducible classes in $\cat{D}_\X$ (resp. in $\cat{D}_\Y$), and $F_{t r} = \Mor(Y_t, F(X_r))$ for $(t, r) \in J' \times J$ be the vector spaces associated with $\theta_\#$.  In terms of the vector spaces $(F_{t r})_{r \in J, t \in J'}$, $(F^\X_{r s})_{r, s \in J}$, and $(F^\Y_{t u})_{t, u \in J}$, the above means that the unitary maps
\[
\psi_{u, t, s} \colon \bigoplus_{r \in J} F_{t r} \otimes F^\X_{r s} \rightarrow \bigoplus_{q \in J'} F^\Y_{t q} \otimes F_{q s}
\]
for $s \in I$ make the diagram
\begin{equation}
\label{EqRepSU2HomFundSolCond}
\xymatrix{
F_{t r} \ar[r]^-{\id \otimes \Rmat} \ar[d]_{\Rmat \otimes \id} & \oplus_s F_{t r} \otimes F^\X_{r s} \otimes F^\X_{s r} \ar[d]^{\id \otimes \psi \circ \psi \otimes \id} \\
\oplus_u F^\Y_{t u} \otimes F^\Y_{u t} \otimes F_{t r} \ar@{^{(}->}[r]& \oplus_{u, s} F^\Y_{t u} \otimes F^\Y _{u s} \otimes F_{s r}
}
\end{equation}
commutative.

The next proposition shows that any $\Rep(\SU_q(2))$-homomorphism functor can be characterized by this commutativity.

\begin{Theorem}\label{ThmSUq2EqvHom}
Let $G$ be a $^*$-preserving functor from $\cat{D}_\X$ to $\cat{D}_\Y$ represented by the $J' \times J$-graded vector space $(F_{t r})_{t \in J, r \in J}$.  Suppose that there exists a unitary map
\begin{equation}\label{EqIntwUntry}
\psi_{t,r}\colon \bigoplus_{s \in J} F_{t s} \otimes F^\X_{s r} \rightarrow \bigoplus_{u \in J'} F^\Y_{t u} \otimes F_{u r}
\end{equation}
for each $t$ and $r$, such that the diagram~\eqref{EqRepSU2HomFundSolCond} is commutative for any $t$ and $r$.  Then $\psi$ can be extended so that $(G, \psi)$ becomes an $\Rep(\SU_q(2))$-homomorphism.
\end{Theorem}

\begin{proof}
The assertion implies that there is a natural equivalence $\psi\colon M(u, G -) \rightarrow G M(u, -)$ of functors from $\cat{D}_\X$ to $\cat{D}_\Y$ such that the diagram
\begin{equation}
\label{EqCommDiagFundActionAndG}
\xymatrix{
G x \ar[d]_{\Rmat_{G x}} \ar[r]^{G(\Rmat_{\X})} & G M(u \otimes u, x) \ar[dr]^{G(\phi_{u,u,x})} &\\
M(u \otimes u, G x) \ar[r]_{\phi_{u, u, G x}} & M(u, M(u, G x)) \ar[r]_{\psi \circ \psi} & G M (u, M(u, x))
}
\end{equation}
is commutative.

We first define the isomorphisms
\[
\psi_{u^{\otimes k}, x} \colon M(u^{\otimes k}, G x) \rightarrow G M(u^{\otimes k}, x)
\]
for $k \in \N$ and $x \in \cat{D}_\X$ by
\[
\psi_{u^{\otimes k}, x} = G(\phi^{-k}) \circ \psi \circ M(\psi, M(u, \cdots, x)\cdots) \circ \cdots \circ M(u, M(u \cdots, \psi) \cdots ) \circ \phi^k.
\]

Next, the assumption on $(G, \psi)$ implies that any morphism $f\colon u^{\otimes k} \rightarrow u^{\otimes k + 2}$ in $\Rep(\SU_q(2))$ of the form $\Rmat_{j, j+1}$ satisfies $f \psi_{u^{\otimes k}, x} = \psi_{u^{\otimes k+2}, x} f$ for arbitrary $x$.  Applying the $^*$-operation to the diagram~\eqref{EqCommDiagFundActionAndG} and using the fact that $\psi$ is unitary, we obtain the same statement for the morphisms $\Rmat_{j, j+1}^*$.  Since the morphisms of the form $\Rmat_{j, j+1}$ and $\Rmat_{j, j+1}^*$ for $j \in \N$ generate the full subcategory of $\Rep(\SU_q(2))$ with the objects $(u^{\otimes k})_{k \in \N}$, we conclude that $f \psi_{u^{\otimes k}, x} = \psi^{u^{\otimes k'}, x} f$ for arbitrary $f \in \Mor(u^{\otimes k}, u^{\otimes k'})$ and $x \in \cat{D}_\X$.

Now a standard argument shows that $\psi$ can be extended to $\Rep(\SU_q(2))$.
\end{proof}

By virtue of the above theorem, we can reduce the problem of finding equivariant morphisms between quantum homogeneous spaces to solving a series of quadratic equations on the set of unitaries.  For example, this gives another way to classify the coideals of $\SU_q(2)$ which is easier and more conceptual than the `generators and relations' method of Podle\'{s} and Tomatsu.  In some parts of Tomatsu's classification, various consideration about the (non)existence of coideals with particular graphs needed to be combined each other to say whether a coideal of a particular diagram exists or not.  In our approach, in order to determine the existence of a particular type of coideal, we only need to study edge weights and the quadratic equation on unitaries for that particular graph.

\begin{Exa}\label{ExaPodlesSphEmb}
Let $0 < |q| < 1$.  The embedding of a Podle\'{s} sphere (Example~\ref{ExaPodlesSphGraphs}) into $\CoL{\SU_q(2})$ can be explained as follows.  First,  if there is any such an embedding $\theta$, the Hilbert spaces $(F_m)_{m \in \Z}$ associated with it have to be $1$-dimensional by~\cite[Proposition~7.5]
{DCY2}.  Thus, we need to find unitaries
\[
\psi_m\colon  ( F_{m + 1} \otimes H_{m + 1,m} ) \oplus ( F_{m - 1} \otimes H_{m - 1,m} ) \rightarrow \Hsp_{1/2} \otimes F_m
\]
for $m \in \Z$ satisfying the compatibility condition of Theorem~\ref{ThmSUq2EqvHom}.  The $q$-fundamental solution of the Podle\'{s} sphere for the parameter $x$ can be represented by
\[
- \absv{\frac{q^{x+m+1} + q^{-(x+m+1)}}{q^{x+m} + q^{-(x+m)}}}^{\frac{1}{2}} \xi_{m, m+1} \otimes \xi_{m+1, m} + \sgn(q) \absv{\frac{q^{x+m-1} + q^{-(x+m-1)}}{q^{x+m} + q^{-(x+m)}}}^{\frac{1}{2}} \xi_{m, m - 1} \otimes \xi_{m - 1, m}
\]
where $\xi_{m, m \pm 1}$ are unit vectors in $H_{m, m \pm 1}$.  Fix a unit vector $\xi_m$ in $F_m$ for each $m \in \Z$.  First, let us consider the case $0 < q$.  Comparing the above with the $q$-fundamental solution of $u$ in~\eqref{EqRmatForHalfSpin}, we deduce that $\psi$ has to be of the form
\begin{align*}
\psi(\xi_{m+1} \otimes \xi_{m+1, m}) &= (-1)^m \alpha_m \bar{\alpha}_{m+1} \left ( \lambda \absv{\frac{q^{-(x+m)}}{q^{x+m} + q^{-(x+m)}}}^{\frac{1}{2}} e_2 \otimes \xi_m + \absv{\frac{q^{x+m}}{q^{x+m} + q^{-(x+m)}}}^{\frac{1}{2}} e_1 \otimes \xi_m \right ), \\
\psi(\xi_{m-1} \otimes \xi_{m-1, m}) &= (-1)^m \alpha_m \bar{\alpha}_{m-1} \left ( - \lambda \absv{\frac{q^{x+m}}{q^{x+m} + q^{-(x+m)}}}^{\frac{1}{2}} e_2 \otimes \xi_m + \absv{\frac{q^{-(x+m)}}{q^{x+m} + q^{-(x+m)}}}^{\frac{1}{2}} e_1 \otimes \xi_m \right )
\end{align*}
(if $x = \infty$, we take $(-1)^m \alpha_m \bar{\alpha}_{m+1} \lambda e_2 \otimes \xi_m$ and $(-1)^m \bar{\alpha}_m \alpha_{m-1} e_1 \otimes \xi_m$ respectively) for some unit modulus complex numbers $(\alpha_m)_{m \in \Z}$ and $\lambda$.  The case of $q < 0$ is almost the same, given by
\begin{align*}
\psi(\xi_{m+1} \otimes \xi_{m+1, m}) &= \alpha_m \bar{\alpha}_{m+1} \left ((-1)^m \lambda \absv{\frac{q^{-(x+m)}}{q^{x+m} + q^{-(x+m)}}}^{\frac{1}{2}} e_2 \otimes \xi_m + \absv{\frac{q^{x+m}}{q^{x+m} + q^{-(x+m)}}}^{\frac{1}{2}} e_1 \otimes \xi_m \right ), \\
\psi(\xi_{m-1} \otimes \xi_{m-1, m}) &= \alpha_m \bar{\alpha}_{m-1} \left ( \lambda \absv{\frac{q^{x+m}}{q^{x+m} + q^{-(x+m)}}}^{\frac{1}{2}} e_2 \otimes \xi_m - (-1)^m \absv{\frac{q^{-(x+m)}}{q^{x+m} + q^{-(x+m)}}}^{\frac{1}{2}} e_1 \otimes \xi_m \right ).
\end{align*}
Taking into account of the gauge action (see~\cite[Corollary~7.4]
{DCY2}), we see that there is an free transitive action of $U(1)$ on the embeddings of $\Co{S^2_{q, x}}$ if $x \neq \infty$ (given by the right translation of $U(1)$), and that the embedding is unique for $x = \infty$.
\end{Exa}

\begin{Exa}\label{ExaQRProjEmb}
We can see that category of the graph $D_\infty$ (Example~\ref{ExaQRProjSp}) admits an (essentially unique) $R(\SU_q(2))$-homomorphism functor to the one of $A_{\infty, \infty}$ (Example~\ref{ExaPodlesSphGraphs}) with the parameter $x$ if and only if $x=0$.  Then, the endpoints of $D_\infty$ are mapped to the vertex $0$ of $A_{\infty, \infty}$.  In fact, most of the unitaries to consider for~\eqref{EqIntwUntry} are between $1$-dimensional Hilbert spaces.  This excludes the cases $0 < x$, and also forces the endpoints to be mapped to the vertex $0$.   For $x=0$, the condition forces $\theta_\#(X_m) \simeq Y_{-m} \oplus Y_{m}$ for $1 \le m$, and $\theta_\#(X_*) = \theta_\#(X_{\tilde{*}}) \simeq Y_0$.  Hence $\theta_\#$ is represented by the $1$-dimensional spaces $F_*$, $F_{\tilde{*}}$, $F_m$, $F_{-m}$ for $r = 1, 2, \ldots$, One can verify that the unitary map from $F_* \otimes H_{* 1} \oplus F_{\tilde{*}} \otimes H_{\tilde{*} 1}$ to $H_{0 1} \otimes F_1 \oplus H_{0, -1} \otimes F_{-1}$ expressed by the matrix
\[
\frac{1}{\sqrt{2}}
\begin{pmatrix}
1 & 1\\
-1 & 1
\end{pmatrix}
\]
satisfies the condition of $\psi$. This solution corresponds to the embedding of $C(\R P_{q, \infty}^2)$ into $C(S_{q, 0}^2)$.
\end{Exa}

\begin{Exa}\label{ExaCoidTypeAinf'}
Let $q < 0$, and consider the graph of type $A_\infty'$ (Example~\ref{ExaGraphAinfty'}).  By an analogous argument as above, we see that the Podle\'{s} sphere for $x = \half$ embeds into this one, by the correspondence of vertices $m \mapsto m$ for $m \ge 0$ and $m \mapsto -(m+1)$ for $m < 0$.  A formula similar to the ones of Example~\ref{ExaPodlesSphEmb} gives the embedding of the algebras $(\Co{\X_m})_{m=0}^\infty$ into $\CoL{\SU_q(2)}$.
\end{Exa}

\begin{Exa}\label{ExaQ-1Cids}
Consider the case $q = -1$ and $\X$ is either of the ones in Example~\ref{Exaqminus1AmDn}.  We then find solutions of $(G, \psi)$ for the embeddings $\Co{\X} \rightarrow \CoL{\SU_{-1}(2)}$ coinciding with the results of~\cite{Tom2}.  As an illustration let us consider the graph $D_3'$.  It has three vertices we label by $+, -, *$, where $*$ is the middle one admitting a loop around it.  Since the associated graded vector space has $1$-dimensional components, we pick a unit vector in each component denoted by $\xi_{+ *} \in F_{+ *}$, etc.  Thus, the $q$-fundamental solution can be represented by the vectors
\[
\xi_{* *} \otimes \xi_{* *} + \frac{1}{\sqrt{2}} \xi_{* +} \otimes \xi_{+ *} + \frac{1}{\sqrt{2}} \xi_{* -} \otimes \xi_{- *},\quad \sqrt{2} \xi_{+ *} \otimes \xi_{* +}, \quad \sqrt{2} \xi_{- *} \otimes \xi_{* -}.
\]
The vector spaces $(F_*, F_+, F_-)$ associated with $G\colon \cat{D}_\X \rightarrow \Rep(\SU_q(2))$ has to satisfy $\dim F_* = 2$ and $\dim F_+ = \dim F_- = 1$.  We take a basis $(\xi_1, \xi_2)$ of $F_*$, and unit vectors $\xi_{\pm} \in F_{\pm}$.  Solving the equation for $\psi$, we find that the solutions are the following two families, both parametrized by $\beta \in U(1)$.  The first is
\begin{gather*}
\xi_1 \otimes \xi_{* +} \mapsto e_1 \otimes \xi_+, \quad \xi_2 \otimes \xi_{* +} \mapsto e_2\otimes \xi_+, \quad \xi_+ \otimes \xi_{+ *} \mapsto \frac{1}{\sqrt{2}} (e_2 \xi_1 + e_1 \xi_2),\\
\xi_1 \otimes \xi_{* -} \mapsto e_1 \otimes \xi_-, \quad \xi_2 \otimes \xi_{* -} \mapsto -e_2 \otimes \xi_-, \quad \xi_- \otimes \xi_{- *} \mapsto \frac{1}{\sqrt{2}} (e_2 \xi_1 - e_1 \xi_2),\\
\xi_1 \otimes \xi_{* *} \mapsto \beta e_2 \otimes \xi_2,\quad \xi_2 \otimes \xi_{* *} \mapsto \bar{\beta} e_1 \otimes \xi_1,
\end{gather*}
and the second is
\begin{gather*}
\xi_1 \otimes \xi_{* +} \mapsto e_1 \otimes \xi_+, \quad \xi_2 \otimes \xi_{* +} \mapsto e_2\otimes \xi_+, \quad \xi_+ \otimes \xi_{+ *} \mapsto \frac{1}{\sqrt{2}} (e_2 \xi_1 + e_1 \xi_2),\\
\xi_1 \otimes \xi_{* -} \mapsto -\bar{\beta}^4 e_2 \otimes \xi_-, \quad \xi_2 \otimes \xi_{* -} \mapsto e_1 \otimes \xi_-, \quad \xi_- \otimes \xi_{- *} \mapsto \frac{1}{\sqrt{2}} (-\beta^4 e_1 \xi_1 + e_2 \xi_2),\\
\xi_1 \otimes \xi_{* *} \mapsto \frac{1}{2}((\beta e_1 + \bar{\beta} e_2) \otimes \xi_1 + (-\bar{\beta} e_1 + \bar{\beta}^3 e_2) \otimes \xi_2),\\
\xi_1 \otimes \xi_{* *} \mapsto \frac{1}{2}((\beta^3 e_1 - \beta e_2) \otimes \xi_1 + (\beta e_1 +\bar{\beta} e_2) \otimes \xi_2).
\end{gather*}
\end{Exa}

\section{The C$^*$-algebra of a quantum homogeneous spaces}

It would be very pleasing if certain properties of the $\Co{\X}$ could be deduced directly from their associated fair and balanced $\qn{2}$-graph. For example, can something be said about the type of their associated von Neumann algebras? About their center, or the center of their associated von Neumann algebra? As we will see, there is a satisfying description of their $K$-theory in such terms. For this, we will not need to make use of the cost function. We first however make some comments on nuclearity and C$^*$-algebraic types.

\subsection{C$^*$-algebraic properties of quantum homogeneous spaces}

\begin{Prop} Let $\X$ be a quantum homogeneous space for $\SU_q(2)$, where $0<|q|\leq 1$. Then $\Co{\X}$ is nuclear.\end{Prop}
\begin{proof} This follows from~\cite[Corollary~23]{Boc1}.
\end{proof}

We now turn to a discussion of the type of the $\cC^*$-algebras associated with ergodic actions by $\SU_q(2)$.

Recall that a unital $\cC^*$-algebra is said to be \emph{liminal} if all its irreducible representations are finite-dimensional.  The algebra $C(\SU(2))$ is liminal since it is commutative, and $C(\SU_{-1}(2))$ is also known to be liminal; any of its irreducible representations is at most $2$-dimensional~\cite{Wor3}.

When $\absv{q} < 1$, the $\cC^*$-algebra $\CoL{\SU_q(2)}$ is no longer liminal, but it still is of type I. The following proposition concerns the type of the $\cC^*$-algebra associated with a quantum homogeneous space.

\begin{Prop}\label{PropQHomCAlgType1} Let $\X$ be a quantum homogeneous space for $\SU_q(2)$, where $0<|q|\leq 1$. Then $\Co{\X}$ is of type I if and only if $\Co{\X}$ is equivariantly Morita equivalent to a coideal.
\end{Prop}

\begin{proof}
Since the property of being type I passes to sub-C$^*$-algebras by Glimm's Theorem (cf. \cite[IV.1.5.11]{Bla1}), any coideal for $\SU_q(2)$ is of type I. Since this property is also preserved under strong Morita equivalence, the `if' direction of the assertion is proven.

Conversely, assume that $A$ is of type I. Then $A$ has at least one non-zero finite-dimensional representation. By \cite[Theorem 4.10]{Tom1}, the graph associated with this ergodic action must have norm smaller than or equal to $2$, hence equal to $2$. By the arguments as in Theorem~\ref{TheoBound}, it must be equivariantly Morita equivalent to a coideal.
\end{proof}

\begin{Rem}
Quite possibly this proposition remains true also on the von Neumann algebraic level. Of course one direction is trivial, but we have not been able to produce a proof for the converse direction.
\end{Rem}

\subsection{K-theory of quantum homogeneous spaces}

Let us now consider the $K$-theory of quantum homogeneous spaces. The solution of the Baum--Connes conjecture by Nest and Voigt~\cite{Nes1,Voi1} allows us to compute the (usual) K-groups of quantum homogeneous spaces from the equivariant K-groups.  First, let us briefly review their methodology.  In the $\Z$-equivariant KK-category, there is an exact triangle of the form
\[
\xymatrix{
C_0(\Z) \ar[rr]^{\tau - \id} & & C_0(\Z) \ar[ld]\\
& \mathbb{\C} \ar[lu], &
}
\]
where $\tau$ is the automorphism of $C_0(\Z)$ given as the translation by $1 \in \Z$.    By the Baaj--Skandalis duality and the Takesaki--Takai duality, we obtain an exact triangle
\[
\mathbb{C} \xrightarrow{\chi_1-\chi_0} \mathbb{C} \rightarrow \CoL{U(1}) \rightarrow \mathbb{C}[1]
\]
in the $U(1)$-equivariant KK-category.  The induction from $U(1)$ to $\SU_q(2)$ gives yet another exact triangle involving the standard Podle\'s sphere and $\SU_q(2)$,
\begin{equation*}\label{EqSUq2TriFlag}
\CoL{U(1)\backslash\SU_q(2)} \rightarrow \CoL{U(1)\backslash\SU_q(2)}  \rightarrow \CoL{\SU_q(2}) \rightarrow \CoL{U(1)\backslash\SU_q(2)},
\end{equation*}
in the $\SU_q(2)$-equivariant KK-category.  This triangle can be lifted to a $D(\SU_q(2))$-equivariant exact triangle for the Drinfeld double $D(\SU_q(2))$.  The solution of the Baum--Connes conjecture for $\SU_q(2)$ by Voigt states that $\CoL{U(1)\backslash\SU_q(2)}$ is $D(\SU_q(2))$-equivariantly KK-equivalent to $\C \oplus \C$, hence we obtain an exact triangle of the form
\begin{equation}\label{EqTriDGq}
\C \oplus \C \rightarrow \C \oplus \C \rightarrow \CoL{\SU_q(2}) \rightarrow (\C \oplus \C )[1].
\end{equation}

\begin{Theorem}\label{ThmKGrpComput}
Let $\X$ be a quantum homogeneous space over $\SU_q(2)$, and let $\gamma$ be the action of $\pi_{\frac{1}{2}} \in R(\SU_q(2))$ on $K^{\SU_q(2)}_*(\Co{\X})$.  Let $\phi$ be the map on $K_0^{\SU_q(2)}(\Co{\X})^{\oplus 2}$ represented by the matrix
\[
\begin{pmatrix}-\id & -\id \\ \id & \gamma - \id\end{pmatrix}.
\]
Then $K_0(\Co{\X})$ (resp. $K_1(\Co{\X})$) is the cokernel (resp. kernel) of $\phi$.
\end{Theorem}

\begin{proof}
For any $\G$-$\cC^*$-algebra $A$, the braided tensor product $A \boxtimes -$ is an exact functor from the category of $D(\SU_q(2))$-algebras to that of $\SU_q(2)$-algebras.  Thus, applying it to the triangle~\eqref{EqSUq2TriFlag}, we obtain an exact triangle of $\SU_q(2)$-algebras
\[
(A \boxtimes \C)^{\oplus 2} \rightarrow (A \boxtimes \C)^{\oplus 2} \rightarrow (A \boxtimes \CoL{\SU_q(2}))^{\oplus 2} \rightarrow (A \boxtimes \C)^{\oplus 2}[1].
\]
This induces a $6$-term exact sequence of the $K^{\SU_q(2)}$-groups.

Now, we have a natural identification $A \boxtimes \C \simeq A$ as $\SU_q(2)$-algebras.  Moreover, $A \boxtimes \CoL{\SU_q(2})$ is naturally isomorphic to $\Co{\SU_q(2}) \otimes A_{\triv}$, where $A_{\triv}$ denotes is the $\cC^*$-algebra $A$ with the trivial $\CoL{\SU_q(2})$-coaction.  By the Green--Julg isomorphism and the Takesaki--Takai duality, we have $K^{\SU_q(2)}_*(\CoL{\SU_q(2}) \otimes A_{\triv}) \simeq K_*(A)$.  Furthermore, the Green--Julg isomorphism together with the ergodicity on $A$ implies that $K^{\SU_q(2)}_1(A)$ is trivial.  Thus, the above triangle gives an exact sequence of the form
\[
0 \rightarrow K_1(A) \rightarrow K^{\SU_q(2)}_0(A)^{\oplus 2} \rightarrow K^{\SU_q(2)}_0(A)^{\oplus 2} \rightarrow K_0(A) \rightarrow 0.
\]

It remains to identify the map $f$ between $K^{\SU_q(2)}_0(A)^{\oplus 2}$ in the above exact sequence.  First, $K^{\SU_q(2)}_0(A)^{\oplus 2}$ can be identified with
\[
R(U(1)) \otimes_{R(\SU_q(2))} K^{\SU_q(2)}_0(A) \simeq K^{U(1)}_0(A)
\]
because $R(U(1))$ is free of rank $2$ over $R(\SU_q(2))$.  If one follows the argument above the theorem, we see that $f$ can be identified with the action of $z - 1 \in R(U(1))$ on $K^{U(1)}_0(A)$.  Let us take $(1, z)$ as a basis of $R(U(1))$ over $R(\SU_q(2))$.  Then, by $z^2 = (z + z^{-1}) z - 1$, the action of $z$ on $R(U(1))$ can be expressed by the matrix
\[
\begin{pmatrix}
0 & -1\\
1 & \pi_{\frac{1}{2}}
\end{pmatrix}
\in M_2(R(\SU_q(2))).
\]
This shows that $z - 1$ acts by the matrix of assertion.
\end{proof}

\begin{Exa} Let us consider an ergodic action of $\SU_q(2)$ of full quantum multiplicity, cf. \cite{BDV1}. This means that the associated weighted graph has a single vertex, and its loops can be labeled by a multiset $S$ of positive real numbers $\lambda_i$, of even cardinality when $q>0$, such that $S=S^{-1}$ and such that the sum of all elements in $S$ equals $\aqn{2}$. Let us denote by $F_S$ some real-valued matrix such that $F_S^2 = -\sgn(q)$, and such that $S$ forms the set of eigenvalues of $F^*F$ (taking account of multiplicities). Then the ergodic action associated with our weighted graph is the C$^*$-algebra denoted $A_o(F_S,F)$ in \cite{BDV1} (which up to isomorphism does not depend on the concrete choice of $F_S$, and where we recall that $F$ is the matrix \eqref{EqFmat}).

In this case, we have that $K_0^{\SU_q(2)}(A_o(F_S,F)) = \mathbb{C}$, and the matrix $\phi$ from Theorem \ref{ThmKGrpComput} becomes \[\phi = \begin{pmatrix} -1 & -1 \\ 1 & n-1 \end{pmatrix},\] where $n$ is the number of loops in our graph. It follows that, for $n\neq 2$, the $A_o(F_S,F)$ have Cuntz-algebra-like behaviour in that $K_1(A_o(F_S,F))= 0$ and $K_0(A_o(F_S,F)) = \Zz_{n-2}$.
\end{Exa}

\hrulefill
\begin{center}
Department of Mathematics, University of Cergy-Pontoise, UMR CNRS 8088, F-95000 Cergy-Pontoise, France\\
e-mail: Kenny.De-Commer@u-cergy.fr

Department of Mathematics, Ochanomizu University, Otsuka 2-1-1, Bunkyo, 112-8610, Tokyo, Japan\\
e-mail: yamashita.makoto@ocha.ac.jp
\end{center}
\end{document}